\documentclass[11pt]{amsart}
\usepackage{amsfonts, amsmath, amssymb, color}
\usepackage{amsthm}
\usepackage{booktabs}
\usepackage{float}
\usepackage{hhline}
\usepackage{lipsum}
\usepackage{lscape}
\usepackage{subfig}
\usepackage{textcomp}
\usepackage{tikz}
\usetikzlibrary{decorations.pathmorphing,shapes,arrows,positioning}
\usepackage{xcolor}
\usepackage{young}
\usepackage{youngtab}
\usepackage{ytableau}
\usepackage{rotating} 
\usepackage{hyperref}
\usepackage{scalefnt}
\usetikzlibrary{calc}

\hypersetup{colorlinks,linkcolor=blue,urlcolor=cyan,citecolor=blue}

\allowdisplaybreaks[4]
\theoremstyle{plain}
\newtheorem{thm}{Theorem}[section]
\newtheorem{lem}[thm]{Lemma}
\newtheorem{prop}[thm]{Proposition}
\newtheorem{cor}[thm]{Corollary}
\newtheorem{rem}[thm]{Remark}

\theoremstyle{definition}

\newtheorem{exmp}[thm]{Example}

\makeatletter

\newcommand{\Rmnum}[1]{\expandafter\@slowromancap\romannumeral #1@}
\makeatother
\newcommand{\la}{\lambda}

\numberwithin{equation}{section} \errorcontextlines=0

\begin{document}
\title[A multiparametric Murnaghan-Nakayama rule for Macdonald polynomials]{A multiparametric Murnaghan-Nakayama rule for Macdonald polynomials}
\author{Naihuan Jing, Ning Liu}
\address{Department of Mathematics, North Carolina State University, Raleigh, NC 27695, USA}
\email{jing@ncsu.edu}
\address{Fakult\"at f\"ur Mathematik, Universit\"at Wien, Oskar-Morgenstern-Platz 1, A-1090 Vienna, Austria}
\email{mathliu123@outlook.com}

\maketitle

\begin{abstract}
We introduce a new family of operators as multi-parameter deformation of the one-row Macdonald polynomials. The matrix coefficients of these operators acting on the space of symmetric
functions with rational coefficients in two parameters $q,t$ (denoted by $\Lambda[q,t]$) are computed by assigning some values to skew Macdonald polynomials in $\la$-ring notation. The new rule is utilized to provide new iterative formulas
and also recover various existing formulas in a unified manner.
Specifically the following applications are discussed: (i) A $(q,t)$-Murnaghan-Nakayama rule for Macdonald functions is given as a generalization of the $q$-Murnaghan-Nakayama rule; (ii) An iterative formula for the $(q,t)$-Green polynomial is deduced; (iii) A simple proof of the Murnaghan-Nakayama rule for the Hecke algebra and the Hecke-Clifford algebra is offered; (iv) A combinatorial inversion of the Pieri rule for Hall-Littlewood functions is derived with the help of the vertex operator realization of the Hall-Littlewood functions; (v) Two iterative formulae for the $(q,t)$-Kostka polynomials $K_{\lambda\mu}(q,t)$ are obtained from the dual version of our multiparametric Murnaghan-Nakayama rule, one of which yields an explicit formula for arbitrary $\lambda$ and $\mu$ in terms of the generalized $(q, t)$-binomial coefficient introduced independently by Lassalle and Okounkov.
\end{abstract}
\tableofcontents

\section{Introduction}
The Murnaghan-Nakayama rule \cite{Mur,Nak1,Nak2} is a combinatorial algorithm for calculating irreducible characters of the symmetric groups. A formula for irreducible characters of the Iwahori-Hecke algebra of type A (called $q$-Murnaghan-Nakayama rule) is known \cite{Ram} (also see \cite{KW, Van}) as a $q$-analog of the classical Murnaghan-Nakayama rule. The authors formulated a determinant-type Murnaghan-Nakayama rule \cite{JL2}  for the Hecke algebras using vertex operators, which provided
an alternating proof of the $q$-Murnaghan-Nakayama rule.

As a $q$-deformation of the Sergeev algebra $\mathfrak{H}^c_n$, the Hecke-Clifford algebra $\mathcal{H}^c_n$ was defined by Olshanski \cite{O} as a semidirect product of the Hecke algebra $\mathcal H_n$ by the Clifford algebra $Cl_n$. It is the associative superalgebra over the field $\mathbb{C}(t^{\frac{1}{2}})$ with even generators $T_1,\cdots, T_{n-1}$ and odd generators $c_1,\cdots, c_n$ subject to the relations:
\begin{align*}
(T_i-t)(T_i+1)=0, \quad &1\leq i\leq n-1,\\
T_iT_j=T_jT_i, \quad &1\leq i,j\leq n-1, \mid i-j\mid>1,\\
T_iT_{i+1}T_i=T_{i+1}T_iT_{i+1}, \quad &1\leq i\leq n-2,\\
c^2_i=1, c_ic_j=-c_jc_i, \quad &1\leq i\neq j\leq n,\\
T_ic_j=c_jT_i, \quad &j\neq i, i+1, 1\leq i\leq n-1, 1\leq j\leq n,\\
T_ic_i=c_{i+1}T_i, \quad &1\leq i\leq n-1,
\end{align*}
where the first three relations define the Hecke algebra $\mathcal{H}_n$. In \cite{WW}, Wan and Wang used the super-Schur-Weyl duality to derive a Frobenius character formula for the Hecke-Clifford algebra in terms of spin Hall-Littlewood functions. Recently, the authors \cite{JL3} used the same technique of \cite{JL2} to derive a Pfaffian-type Murnaghan-Nakayama rule for the Hecke-Clifford algebras, which yields the combinatorial Murnaghan-Nakayama rule for the Hecke-Clifford algebras.

Macdonald symmetric functions $Q_{\la}(X;q,t)$ include several classes of classical symmetric functions as special cases such as Jack functions $Q_{\la}(X;\alpha)$, Hall-Littlewood functions $Q_{\la}(X;0,t)$, Schur functions $Q_{\la}(X;q=t)$ and Schur $Q$-functions $Q_{\la}(X;0,-1)$, etc. The aim of this paper is to give a Murnaghan-Nakayama rule for the Macdonald polynomials, which will specialize to the two formulas above for the Hecke algebras and the Hecke-Clifford algebras. To do this, we introduce a new family of symmetric functions and obtain a multi-parameter Murnaghan-Nakayama rule for the Macdonald polynomials in $\la$-ring notation. Based on our formulae, an iterative formula for the $(q,t)$-Green polynomials is derived,
and an inversion of the Pieri rule for Hall-Littlewood functions is deduced combining with the vertex operator realization of the Hall-Littlewood functions.

One of the most important problems in algebraic combinatorics is to find the combinatorial expansion of the integral Macdonald polynomials $J_{\mu}(X; q, t)$ in terms of the modified Schur functions $S_{\la}(X; t)$. Much efforts have been devoted in searching of positive combinatorial
formulas for $(q, t)$-Kostka polynomials $K_{\la\mu}(q, t)$ \cite{A, B, HHL, Ste, Z}, and recently a direct formula of $K_{\la\mu}(q, t)$ was found
by Gaballi and Wheeler \cite{GW} using colored paths on a square lattice with boundary conditions. Last but not least, we use
the multiparameter Murnaghan-Nakayama rule to give two direct and analytic formulae for the $(q, t)$-Kostka polynomials $K_{\la\mu}(q, t)$.

Our first direct formula of $K_{\la\mu}(q, t)$ employs the $(q, t)$-binomial coefficients introduced by
Lassalle \cite{L2} and Okounkov \cite{Ok} independently. The $q$-binomial theorem \cite{GR} is essentially monomial expansion of
$(x-1)(x-q)\cdots(x-q^{k-1})$. In \cite{Ok}, to obtain its multivariate generalization in the context of symmetric Macdonald polynomials, Okounkov introduced the generalized $(q, t)$-binomial coefficient $\left(\begin{matrix}\la\\ \mu\end{matrix}\right)_{q,t}$. Using our first iterative formula for the $(q,t)$-Kostka polynomial $K_{\la\mu}(q,t)$, we obtain an explicit formula computing $K_{\la\mu}(q,t)$:
\begin{align}
K_{\la\mu}(q,t)=(-1)^{|\la|}t^{n(\mu)}\sum_{\underline{\la}, \underline{\mu}}(-1)^{wid(\underline{\la})}\prod_{i=0}^{r-1}\left(\begin{matrix}\mu^{(i)}\\ \mu^{(i+1)}\end{matrix}\right)_{q,t},
\end{align}
where $\underline{\la}, \underline{\mu}$ are certain sequences of partitions and $wid(\underline{\la})$ is the total width of the sequence defined explicitly in Sec. 6.

The layout of this paper is as follows. In the next section, we review briefly some terminologies about symmetric functions and $\la$-ring notation will used in the subsequent sections. In Section 3, we give the multi-parameter Murnaghan-Nakayama rule (Corollary \ref{t:M-M-N}) in the $\la$-ring notation generalising the Murnaghan-Nakayama rule (Corollary \ref{t:M-M-N2}) for the Macdonald polynomials. As an application, a recurrence formula for the $(q,t)$-Green polynomials is given (Theorem \ref{t:iteGreen}). Some special cases of our Murnaghan-Nakayama rule are considered in Section 4. We show how special cases of our rule can recover the existing formulae (the Murnaghan-Nakayama rule for the Hecke algebras and the Hecke-Clifford algebras), which explains why we call our formula the Murnaghan-Nakayama rule for the Macdonald polynomials. Section 5 deals with the inversion of the Pieri rule for the Hall-Littlewood functions (Theorem \ref{t:inverse}). In Section 6, we offer two iterative formulae for the $(q, t)$-Kostka polynomials (Theorem \ref{t:q,t-Kite} and Theorem \ref{t:q,t-Kite2}). The first iterative formula implies a general formula expressing the $(q, t)$-Kostka polynomials in terms of the generalized $(q, t)$-binomial coefficients (Corollary \ref{t:Kgeneral}).

\section{Symmetric functions and $\lambda$-ring notation}
In this section, we review briefly some basic terminologies. The standard reference for symmetric functions and Macdonald polynomials is 
\cite[Chaps. I, VI]{Mac} 
and that for $\lambda$-rings is \cite{K}.
\subsection{Partitions and Macdonald polynomials}
A {\it partition} is a sequence of non-negative integers in decreasing order. The non-zero $\la_i$ are called the {\it parts} of $\la$. The number of the parts is called the {\it length} of $\la$, denoted by $l(\la)$, and the sum of the parts is the {\it weight} of $\la$, denoted by $|\la|$. For convenience, we call
$\la_1$ the {\it width} of the partition $\la$. If $|\la|=n$, we say $\la$ is a partition of $n$, denoted by $\la\vdash n$. $\mathcal{P}_n$ denotes the set of all partitions of $n$. We also use the multiplicity notation for a partition $\lambda$, i.e., $\la=(1^{m_1}2^{m_2}\cdots r^{m_{r}}\cdots)$ and $m_i=m_{i}(\la)={\rm Card}\{j:\la_j=i\}$.  The {\it diagram} of a partition $\la$ is defined as the set of squares such that the $i$-th row has $\la_i$ cells, $i=1,2,\cdots,l(\la)$. The {\it conjugate} of a partition $\la$ is the partition $\la'$ whose diagram is the transpose of the diagram $\la$, i.e., $\la_{i}'={\rm Card}(j:\lambda_j\geq i)$. For instance, $\la=(4,3,1)$, then $\la'=(3,2,2,1)$. Their Young diagrams can be respectively presented as follows:
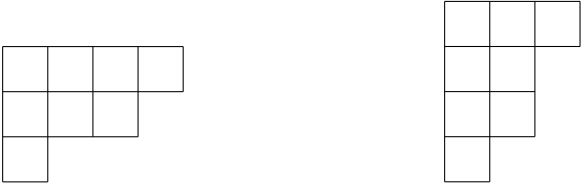
\begin{gather*}
  \centering
\begin{tikzpicture}[scale=0.6]
   \coordinate (Origin)   at (0,0);
    \coordinate (XAxisMin) at (0,0);
    \coordinate (XAxisMax) at (4,0);
    \coordinate (YAxisMin) at (0,-3);
    \coordinate (YAxisMax) at (0,0);
\draw [thin, black] (0,0) -- (4,0);
    \draw [thin, black] (0,-1) -- (4,-1);
    \draw [thin, black] (0,-2) -- (3,-2);
    \draw [thin, black] (0,-3) -- (1,-3);
    \draw [thin, black] (0,0) -- (0,-3);
    \draw [thin, black] (1,0) -- (1,-3);
    \draw [thin, black] (2,0) -- (2,-2);
    \draw [thin, black] (3,0) -- (3,-2);
    \draw [thin, black] (4,0) -- (4,-1);
    \end{tikzpicture}\quad \quad\quad\quad\quad\quad\quad\quad\quad
\begin{tikzpicture}[scale=0.6]
   \coordinate (Origin)   at (0,0);
    \coordinate (XAxisMin) at (0,0);
    \coordinate (XAxisMax) at (4,0);
    \coordinate (YAxisMin) at (0,-3);
    \coordinate (YAxisMax) at (0,0);
\draw [thin, black] (0,0) -- (3,0);
    \draw [thin, black] (0,-1) -- (3,-1);
    \draw [thin, black] (0,-2) -- (2,-2);
    \draw [thin, black] (0,-3) -- (2,-3);
    \draw [thin, black] (0,-4) -- (1,-4);
    \draw [thin, black] (0,0) -- (0,-4);
    \draw [thin, black] (1,0) -- (1,-4);
    \draw [thin, black] (2,0) -- (2,-3);
    \draw [thin, black] (3,0) -- (3,-1);
    \end{tikzpicture}
\end{gather*}
We say $\mu\subset\la$ if $\mu_i\leq\la_i$ for any $i$. If $\mu\subset\la$, the set-theoretic difference $\theta=\la-\mu$ is called a {\it skew diagram}. For example, if $\la=(4,3,3,1)$ and $\mu=(3,2,1)$, the skew diagram $\theta=\la/\mu$ is the shaded region in the picture below.
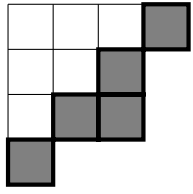
\begin{gather*}
  \centering
\begin{tikzpicture}[scale=0.6]
   \coordinate (Origin)   at (0,0);
    \coordinate (XAxisMin) at (0,0);
    \coordinate (XAxisMax) at (4,0);
    \coordinate (YAxisMin) at (0,-3);
    \coordinate (YAxisMax) at (0,0);
\draw [thin, black] (0,0) -- (4,0);
    \draw [thin, black] (0,-1) -- (4,-1);
    \draw [thin, black] (0,-2) -- (3,-2);
    \draw [thin, black] (0,-3) -- (3,-3);
    \draw [thin, black] (0,-4) -- (1,-4);
    \draw [thin, black] (0,0) -- (0,-4);
    \draw [thin, black] (1,0) -- (1,-4);
    \draw [thin, black] (2,0) -- (2,-3);
    \draw [thin, black] (3,0) -- (3,-3);
    \draw [thin, black] (4,0) -- (4,-1);
    \filldraw[fill = gray][ultra thick]
    (3,0) rectangle (4,-1) (2,-1)rectangle(3,-2) (2,-2)rectangle(3,-3) (1,-2)rectangle(2,-3) (0,-3)rectangle(1,-4);
    \end{tikzpicture}
\end{gather*}

For a skew diagram $\la/\mu$, we define (cf. \cite{KL})
\begin{equation}\label{e:n}
n(\la/\mu)=\sum_i\binom{\la'_i-\mu'_i}{2}.
\end{equation}
A {\it path} in a skew diagram $\theta$ is a sequence $x_0,x_1,\cdots,x_n$ of squares in $\theta$ such that $x_{i-1}$ and $x_i$ have a common side, for $1\leq i\leq n.$ A subset $\xi$ of $\theta$ is connected if any two squares in $\xi$ are connected by a path in $\xi$. The connected components, themselves skew diagrams, are by definition the maximal connected subsets of $\theta$.
A skew diagram $\lambda/\mu$ is a {\it vertical} (resp. {\it horizontal}) {\it strip} if each row (resp. column) contains at most one box. A skew diagram $\theta$ is a {\it border strip} if it contains no $2\times 2$ blocks of squares and is connected (see \cite{Mac}).

Let $\Lambda$ be the ring of symmetric functions in the $x_n$ ($n\in\mathbb N$) over $\mathbb{Q}$, and we will also study the ring $\Lambda_{\mathbb{Z}}$ as a lattice of $\Lambda$. The ring $\Lambda$ has several linear bases indexed by partitions. For each $r>1$, let $p_{r}=\sum x_{i}^{r}$ be the $r$th power-sum. Then $p_{\lambda}=p_{\lambda_{1}}p_{\lambda_{2}}\cdots p_{\lambda_{l}}, \lambda\in\mathcal P$ form a $\mathbb Q$-basis of $\Lambda$.

Let $q,t$ be independent indeterminates and $F=\mathbb{Q}(q,t)$ be the field of rational functions in $q$ and $t$. $\Lambda_{F}=\Lambda\otimes_{\mathbb{Z}}F$ denotes the $F$-algebra of symmetric functions with coefficients in $F$. Using the degree gradation, $\Lambda_{F}$ becomes a graded ring
\begin{align}
\Lambda_{F}=\bigoplus_{n=0}^{\infty} \Lambda_{F}^{n}.
\end{align}
It is well known that the Macdonald polynomials $P_{\la}(X;q,t)$ (or $Q_{\la}(X;q,t)$) ($\la\vdash n$) form an basis of $\Lambda_{F}^{n}$.

Denote
\begin{align*}
z_{\la}(q,t)=z_{\la}\prod_{i=1}^{l(\la)}\frac{1-q^{\la_i}}{1-t^{\la_{i}}}
\end{align*}
where $z_{\lambda}=\prod_{i\geq 1}i^{m_i(\lambda)}m_i(\lambda)!$. The infinite (finite) $q$-shifted factorial in base $q$ can be defined by
\begin{align*}
(x;q)_{\infty}=\prod_{k=0}^{\infty}(1-xq^k), \quad (x;q)_{n}=\prod_{k=0}^{n-1}(1-xq^k).
\end{align*}

Introduce inner product on $\Lambda_{F}$ as follows
\begin{align}\label{e:innerprod}
\langle p_{\la}, p_{\mu} \rangle_{q,t}=\delta_{\la,\mu}z_{\la}(q,t).
\end{align}

For a positive integer $n$, the element $p_n$ acts on the space $\Lambda_{F}$ as the multiplication by $p_n$. Its adjoint operator $p_n^{\perp}$ is the differential operator $\frac{n(1-q^n)}{1-t^n}\frac{\partial}{\partial p_n}$.

Let $Q_{\la}(X;q,t)=b_{\la}(q,t)P_{\la}(X;q,t)$, so that the bases $\{P_{\la}(X;q,t)\}$ and $\{Q_{\la}(X;q,t)\}$ are dual to each other for the inner product, where $b_{\la}(q,t)=\langle P_{\la}(X;q,t), P_{\la}(X;q,t)\rangle_{q,t}^{-1}$ is given by
\begin{align*}
b_{\la}(q,t)=\prod_{1\leq i\leq j\leq l(\la)}\frac{(q^{\la_i-\la_j}t^{j-i+1};q)_{\la_j-\la_{j+1}}}{(q^{\la_i-\la_j+1}t^{j-i};q)_{\la_j-\la_{j+1}}}.
\end{align*}
We will write $P_{\la}$ (resp. $Q_{\la}$) for $P_{\la}(X;q,t)$ (resp. $Q_{\la}(X;q,t)$) in short. Sometimes we also use the notations $P_{\la}(X)$ or $Q_{\la}(X)$ to emphasize the indeterminates.

Let $\la,\mu,\nu$ be any three partitions, define
\begin{align}
f^{\la}_{\mu\nu}=f^{\la}_{\mu\nu}(q,t)=\langle Q_{\la}, P_{\mu}P_{\nu} \rangle_{q,t}.
\end{align}
Equivalently, $f^{\la}_{\mu\nu}$ is the coefficient of $P_{\la}$ in the product $P_{\mu}P_{\nu}$, i.e.,
\begin{align}\label{e:PP}
P_{\mu}P_{\nu}=\sum_{\la}f^{\la}_{\mu\nu}P_{\la}.
\end{align}

Now let $\mu\subset\la$ be two partitions and define $Q_{\la/\mu}$ by
\begin{align}\label{e:la/mu}
Q_{\la/\mu}=\sum_{\nu}f^{\la}_{\mu\nu}Q_{\nu},
\end{align}
which yields
\begin{align*}
\langle Q_{\la/\mu}, P_{\nu} \rangle_{q,t}=\langle Q_{\la}, P_{\mu}P_{\nu} \rangle_{q,t}.
\end{align*}

Let $X=(x_1,x_2,\cdots)$ and $Y=(y_1,y_2,\cdots)$ be two finite or infinite sequences of independent indeterminates, define
\begin{align}\label{e:defCauchy}
\prod(X,Y;q,t)=\prod_{i,j}\frac{(tx_iy_j;q)_{\infty}}{(x_iy_j;q)_{\infty}}.
\end{align}

\begin{lem}\label{t:product}
(\cite[p. 310]{Mac}) Given $n\geq0$, if $(u_{\la})$, $(v_{\la})$ are two $F$-bases of $\Lambda_{F}^{n}$, indexed by partitions of $n$. Then $\langle u_{\la}, v_{\mu} \rangle_{q,t}=\delta_{\la\mu}$ is equivalent to $\sum_{\la}u_{\la}(X)v_{\la}(Y)=\prod(X,Y;q,t)$. In particular, $\sum_{\la}P_{\la}(X)Q_{\la}(Y)=\prod(X,Y;q,t)$.
\end{lem}

\subsection{$\la$-ring notations}
We think of an alphabet $X$ as a sum of commuting variables, i.e., $X_n=x_1+x_2+\cdots+x_n$ is the set of commuting variables $\{x_1, x_2, \cdots, x_n\}.$ From this point of view we can use the following notations:
\begin{align*}
X_n&=\{x_1,x_2,\cdots,x_n\},\\
Y_m&=\{y_1,y_2,\cdots,y_m\},\\
X_nY_m&=\{x_iy_j\}_{1\leq i\leq n, 1\leq j\leq m},\\
X_n+Y_m&=\{x_1,x_2,\cdots,x_n,y_1,y_2,\cdots,y_m\}.
\end{align*}
Moreover, let $-X$ denote a formal (anti-)alphabet such that $X+(-X)=0$.

We have the following properties of the power sum functions:
\begin{align}\label{e:power}
\begin{split}
p_r(X+Y)&=p_r(X)+p_r(Y),\\
p_r(XY)&=p_r(X)p_r(Y),\\
p_r(-X)&=-p_{r}(X).
\end{split}
\end{align}

Using $\la$-ring notations, we can rewrite \eqref{e:defCauchy} as follows:
\begin{align*}
\prod(XY;q,t)=\prod_{i,j}\frac{(tx_iy_j;q)_{\infty}}{(x_iy_j;q)_{\infty}}.
\end{align*}

\begin{lem}
Let $X, Y, Z$ be three alphabets, we have
\begin{align}
\label{e:propCauchy}
\begin{split}
(i)\quad \prod((X+Y)Z;q,t)&=\prod(XZ+YZ;q,t)=\prod_{i,j}\frac{(tx_iz_j;q)_{\infty}}{(x_iz_j;q)_{\infty}}
\prod_{k,l}\frac{(ty_kz_l;q)_{\infty}}{(y_kz_l;q)_{\infty}},\\
(ii)\quad \prod((X-Y)Z;q,t)&=\prod(XZ-YZ;q,t)=\prod_{i,j}\frac{(tx_iz_j;q)_{\infty}}{(x_iz_j;q)_{\infty}}
\prod_{k,l}\frac{(y_kz_l;q)_{\infty}}{(ty_kz_l;q)_{\infty}}.
\end{split}
\end{align}
\end{lem}
\begin{proof}
Note that
\begin{align*}
\prod(XY;q,t)=\prod_{i,j}\frac{(tx_iy_j;q)_{\infty}}{(x_iy_j;q)_{\infty}}=\mbox{exp} \left(\sum_{n\geq1}\frac{1}{n}\frac{1-t^n}{1-q^n}p_n(XY)\right).
\end{align*}
Then \eqref{e:propCauchy} follows from \eqref{e:power}.
\end{proof}

\section{Murnaghan-Nakayama rule with multiple parameters}
In \cite{JL1}, we obtained a Murnaghan-Nakayama rule for Hall-Littlewood functions, with which an iterative formula for the Green polynomial is given. In this section, we give the corresponding generalization: a Murnaghan-Nakayama rule for Macdonald functions and an iterative formula for the $(q, t)$-Green polynomial.
\subsection{Murnaghan-Nakayama rule}
Let $a$ be an indeterminate, $(a;q)_{\infty}$ denotes the infinite product
\begin{align*}
(a;q)_{\infty}=\prod_{r\geq0}(1-aq^r)
\end{align*}
regarded as a formal power series in $a$ and $q$.

We introduce the one-row Macdonald function $g_{r}(X;q,t)$ defined by the following generating function:
\begin{align}\label{e:defg}
\mathbf{g}(X;z)=\sum_{n\geq0}g_{n}(X;q,t)z^n=\prod_{i\geq1}\frac{(tx_iz;q)_{\infty}}{(x_iz;q)_{\infty}}=\mbox{exp} \left(\sum_{n\geq1}\frac{1}{n}\frac{1-t^n}{1-q^n}p_n(x)z^n\right).
\end{align}

Explicitly, we have
\begin{align*}
g_{n}(X;q,t)=\sum_{\la\vdash n}\frac{p_{\la}(X)}{z_{\la}(q,t)},\quad g_{n}(X;q,t)=Q_{(n)}(X;q,t)=\frac{(t;q)_{n}}{(q;q)_{n}}P_{(n)}(X;q,t).
\end{align*}
It is often written in the language of $\la$-rings \cite[p. 223]{L}:
\begin{align*}
g_{k}(X;q,t)=h_k\left(\frac{1-t}{1-q}X\right),
\end{align*}
where $h_k$ is the complete symmetric function.

Let $\mathbf{a}=(a_1,a_2,\cdots,a_s,\cdots)$ and $\mathbf{b}=(b_1,b_2,\cdots,b_r,\cdots)$ be two sequences of finite (or infinite) non-negative parameters. Moreover, denote $\mathbf{a/b}=(a_1,\cdots,a_s,\cdots,-b_1,\cdots,-b_r,\cdots)$. Now we introduce new symmetric functions $g_{n}(\mathbf{a/b};X;q,t)$ with multiparameters defined by
\begin{align}\label{e:defg(a/b)}
\mathbf{g}(\mathbf{a/b};X;z)=\sum_{n\geq0}g_{n}(\mathbf{a/b};X;q,t)z^n=\frac{\prod_{j=1}^{\infty}\mathbf{g}(X;a_jz)}{\prod_{k=1}^{\infty}\mathbf{g}(X;b_kz)}
=\prod_{i\geq1}\prod_{j=1}^{\infty}\prod_{k=1}^{\infty}\frac{(a_jtx_iz;q)_{\infty}}{(a_jx_iz;q)_{\infty}}\frac{(b_kx_iz;q)_{\infty}}{(b_ktx_iz;q)_{\infty}}.
\end{align}
By the last identity of \eqref{e:defg}, $\mathbf{g}(\mathbf{a/b};X;z)$ can be rewritten as
\begin{align*}
\mathbf{g}(\mathbf{a/b};X;z)&=\mbox{exp} \left(\sum_{i=1}^{\infty}\sum_{n\geq1}\frac{1}{n}\frac{1-t^n}{1-q^n}p_n(x)a_i^nz^n-\sum_{j=1}^{\infty}\sum_{n\geq1}\frac{1}{n}\frac{1-t^n}{1-q^n}p_n(x)b_j^nz^n\right)\\
&=\mbox{exp} \left(\sum_{n\geq1}\frac{1}{n}\frac{1-t^n}{1-q^n}p_n(x)p_n(\mathbf{a})z^n-\sum_{n\geq1}\frac{1}{n}\frac{1-t^n}{1-q^n}p_n(x)p_n(\mathbf{b})z^n\right).
\end{align*}

In $\la$-ring language, it follows from \eqref{e:power} that
\begin{align*}
\mathbf{g}(\mathbf{a/b};X;z)&=\mbox{exp} \left(\sum_{n\geq1}\frac{1}{n}\frac{1-t^n}{1-q^n}p_n((\mathbf{a}-\mathbf{b})X)z^n\right)=\mathbf{g}((\mathbf{a}-\mathbf{b})X;z).
\end{align*}
Subsequently,
\begin{align*}
g_{k}(\mathbf{a/b};X;q,t)=g_{k}((\mathbf{a-b})X;q,t).
\end{align*}

Note $p_n^{\bot}=\frac{n(1-q^n)}{1-t^n}\frac{\partial}{\partial p_n}$ with respect to \eqref{e:innerprod}. We have the adjoint operators:
\begin{align}\label{e:gdual}
\mathbf{g}^{\bot}((\mathbf{a}-\mathbf{b})X;z)=\mbox{exp} \left(\sum_{n\geq1}\frac{\partial}{\partial p_n(x)}p_n(\mathbf{a})z^n-\sum_{n\geq1}\frac{\partial}{\partial p_n(x)}p_n(\mathbf{b})z^n\right)=\sum_{k=0}^{\infty}g^{\bot}_{k}((\mathbf{a-b})X;q,t).
\end{align}

\begin{thm}\label{t:skew-M-N}
Let $\mu\vdash n$ and $k$ be a non-negative integer, suppose that $\rho\subset\mu$, then we have
\begin{align}\label{e:skew-M-N}
g_{k}((\mathbf{a-b})X;q,t)Q_{\mu/\rho}(X;q,t)=\sum_{\la\vdash n-|\rho|+k}\sum_{\tau\vdash n-|\rho|}f^{\mu}_{\tau\rho}(q,t)P_{\la/\tau}(\mathbf{a-b};q,t)Q_{\la}(X;q,t),
\end{align}
where $P_{\la/\tau}(\mathbf{a-b};q,t)$ employs the $\la$-ring notation.
\end{thm}
\begin{proof}
Let $Y$ be an arbitrary alphabet and $z$ be an indeterminate, then we have
\begin{align*}
&Q_{\mu/\rho}(zX;q,t)\prod(zX,Y;q,t)\\
=&\sum_{\nu}Q_{\mu/\rho}(zX;q,t)Q_{\nu}(Y;q,t)P_{\nu}(zX;q,t) \quad\text{(by Lemma \ref{t:product})}\\
=&\sum_{\nu}\sum_{\tau}f^{\mu}_{\tau\rho}(q,t)Q_{\tau}(zX;q,t)Q_{\nu}(Y;q,t)P_{\nu}(zX;q,t)  \quad\text{(by \eqref{e:la/mu})}\\
=&\sum_{\la,\nu}\sum_{\tau}f^{\mu}_{\tau\rho}(q,t)b_{\tau}(q,t)f^{\la}_{\tau\nu}Q_{\nu}(Y;q,t)P_{\la}(zX;q,t) \quad\text{(by \eqref{e:PP})}\\
=&\sum_{\la,\tau}f^{\mu}_{\tau\rho}(q,t)b_{\tau}(q,t)Q_{\la/\tau}(Y;q,t)P_{\la}(zX;q,t) \quad\text{(by \eqref{e:la/mu})}\\
=&\sum_{\la,\tau}f^{\mu}_{\tau\rho}(q,t)P_{\la/\tau}(Y;q,t)Q_{\la}(zX;q,t).
\end{align*}
Choose $Y=\mathbf{a-b}$. Note that in this case we have
\begin{align*}
\prod(zX,\mathbf{a-b};q,t)=\mathbf{g}((\mathbf{a-b})X;z)=\sum_{r\geq0}g_{r}((\mathbf{a-b})X;q,t)z^r.
\end{align*}
Therefore,
\begin{align*}
Q_{\mu/\rho}(zX;q,t)\left(\sum_{r\geq0}g_{r}((\mathbf{a-b})X;q,t)z^r\right)=\sum_{\la,\tau}f^{\mu}_{\tau\rho}(q,t)P_{\la/\tau}(\mathbf{a-b};q,t)Q_{\la}(zX;q,t).
\end{align*}
Comparing the coefficient of $z^{n-|\rho|+k}$ on both sides gives \eqref{e:skew-M-N}.
\end{proof}

Choosing $\rho=\emptyset$ in Theorem \ref{t:skew-M-N} gives the following Corollary.

\begin{cor}\label{t:M-M-N}
Let $\mu\vdash n$ and $k$ be a non-negative integer, then we have
\begin{align}\label{e:M-M-N}
g_{k}((\mathbf{a-b})X;q,t)P_{\mu}(X;q,t)=\sum_{\la\vdash n+k}Q_{\la/\mu}(\mathbf{a-b};q,t)P_{\la}(X;q,t).
\end{align}
\end{cor}

\begin{rem}
By the orthogonality of Macdonald polynomials, we have the dual version of \eqref{e:M-M-N} as follows:
\begin{align}\label{e:dual}
g^{\bot}_{k}((\mathbf{a-b})X;q,t)Q_{\la}(X;q,t)=\sum_{\mu\vdash |\la|-k}Q_{\la/\mu}(\mathbf{a-b};q,t)Q_{\mu}(X;q,t).
\end{align}
\end{rem}

We are ready to state our first main result: a Murnaghan-Nakayama rule for Macdonald polynomials. Before that we introduce a modification of $g_{k}((a-1)X;q,t)$ by $\tilde{g}_{k}((a-1)X;q,t):=\frac{1}{a-1}g_{k}((a-1)X;q,t).$ Then we have
\begin{align}\label{e:limit}
\lim_{a\rightarrow1}\tilde{g}_{k}((a-1)X;q,t)=\frac{1-t^k}{1-q^k}p_k(X).
\end{align}
\begin{cor}\label{t:M-M-N2}
Let $\mu\vdash n$ and $k$ be a non-negative integer, then we have
\begin{align}\label{e:M-M-N2}
\tilde{g}_{k}((a-1)X;q,t)P_{\mu}(X;q,t)=\sum_{\la\vdash n+k}\frac{Q_{\la/\mu}(a-1;q,t)}{a-1}P_{\la}(X;q,t).
\end{align}
\end{cor}
\begin{proof}
Choose $\mathbf{a}=(a,0,\cdots,0)$ and $\mathbf{b}=(1,0\cdots,0)$ in Corollary \ref{t:M-M-N}.
\end{proof}

Now we will compute the coefficients $Q_{\la/\mu}(a-1;q,t)$ and derive an explicit formula for it. To reach this, we need some notations. For a skew diagram
$\la/\mu$, denote
\begin{align*}
f(u)&=\frac{(tu;q)_{\infty}}{(qu;q)_{\infty}}, \quad \varphi_{\la/\mu}(q,t)=\prod_{1\leq i\leq j\leq l(\la)}\frac{f(q^{\la_i-\la_j}t^{j-i})f(q^{\mu_i-\mu_{j+1}}t^{j-i})}{f(q^{\la_i-\mu_j}t^{j-i})f(q^{\mu_i-\la_{j+1}}t^{j-i})},\\
&sk_{\la/\mu}(q,t)=Q_{\la/\mu}(\frac{1-q/t}{1-t})=t^{n(\la)-n(\mu)}\prod_{i,j=1}^{l(\la)}\frac{(qt^{j-i-1};q)_{\la_i-\mu_j}(qt^{j-i};q)_{\mu_i-\mu_j}}
{(qt^{j-i-1};q)_{\mu_i-\mu_j}(qt^{j-i};q)_{\la_i-\mu_j}},\\
&\psi'_{\la/\mu}(q,t)=\prod_{(i,j)}\frac{(1-q^{\mu_i-\mu_j}t^{j-i-1})(1-q^{\la_i-\la_j}t^{j-i+1})}{(1-q^{\mu_i-\mu_j}t^{j-i})(1-q^{\la_i-\la_j}t^{j-i})}
\end{align*}
where the product is taken over all pairs $(i,j)$ such that $i<j$ and $\la_i=\mu_i$, $\la_j=\mu_j+1$. Note that
$\varphi_{\la/\mu}(q,t)$ (resp. $\psi'_{\la/\mu}(q,t)$) is zero unless $\la/\mu$ is a horizontal (resp. vertical) strip.

First, by the tableaux realization of Macdonald polynomials, we have \cite[p. 345, (7.9)]{Mac}
\begin{align}
Q_{\la/\mu}(a-1;q,t)=\sum_{\mu\subset\nu\subset\la}Q_{\la/\nu}(a;q,t)Q_{\nu/\mu}(-1;q,t).
\end{align}
It is well known that \cite[p. 346, (7.14)]{Mac}
\begin{align}
Q_{\la/\nu}(a;q,t)=
\begin{cases}
\varphi_{\la/\nu}(q,t)a^{|\la/\nu|}&\text{if $\la/\nu$ is a horizontal strip},\\
0&\text{otherwise}.
\end{cases}
\end{align}
Whereas we have \cite{W}
\begin{align}
Q_{\nu/\mu}(-1;q,t)=\sum_{\eta}(-1)^{|\nu/\eta|}t^{|\eta/\mu|}\psi'_{\nu/\eta}(q,t)sk_{\eta/\mu}(q,t),
\end{align}
summed over all partitions $\eta$ such that $\mu\subset\eta\subset\nu$ and $\nu/\eta$ is a vertical strip.

Therefore, we have
\begin{align}\label{e:Q(a-1)}
Q_{\la/\mu}(a-1;q,t)=\sum_{(\eta,\nu)}(-1)^{|\nu/\eta|}a^{|\la/\nu|}t^{|\eta/\mu|}\psi'_{\nu/\eta}(q,t)\varphi_{\la/\nu}(q,t)sk_{\eta/\mu}(q,t),
\end{align}
summed over all pairs of partitions $(\eta,\nu)$ such that $\mu\subset\eta\subset\nu\subset\la$, $\la/\nu$ is a horizontal strip and $\nu/\eta$ is a vertical strip.

\subsection{An iterative formula for the $(q,t)$-Green polynomial}
For each partition $\la$ we need some statistics defined by
\begin{align*}
c_{\la}(q,t)=\prod_{s\in\la}(1-q^{a(s)}t^{l(s)+1}), \quad c'_{\la}(q,t)=\prod_{s\in\la}(1-q^{a(s)+1}t^{l(s)}), \quad c'_{\la}(q,t)=c_{\la'}(q,t)
\end{align*}
where, for a cell $s\in\la$, $a(s)$ and $l(s)$ represent respectively the arm and the leg of $s$ in $\la$, that is the number of cell of $\la$ that are respectively strictly east and north of $s$.

Define the integral form $J_{\la}(X;q,t)$ as
\begin{align}\label{e:J-def}
J_{\la}(X;q,t)=c_{\la}(q,t)P_{\la}(X;q,t)=c'_{\la}(q,t)Q_{\la}(X;q,t).
\end{align}

For each pair of partitions $\la$ and $\mu$, define $(q,t)$-Green polynomials $X^{\la}_{\mu}(q,t)$ by \cite[p. 356]{Mac}
\begin{align}
J_{\la}(X; q, t)=\sum_{\mu} z^{-1}_{\mu}X^{\la}_{\mu}(q,t)p_{\mu}(X;t)
\end{align}
where $p_{\mu}(X,t)=p_{\mu}(X)\prod_{i=1}^{l(\mu)}(1-t^{\mu_i})$.

By duality, we have
\begin{align}\label{e:pP}
p_{\mu}(X)=\sum_{\la}\frac{\prod_{i=1}^{l(\mu)}(1-q^{\mu_i})}{c'_{\la}(q,t)}X^{\la}_{\mu}(q,t)P_{\la}(X;q,t).
\end{align}

To describe the result we need some notations. For each partition $\la=(\la_1, \ldots, \la_l)$, we define that
\begin{align}
\la^{[i]}=(\la_{i+1}, \cdots, \la_l), \qquad i=0, 1, \ldots, l
\end{align}
So $\la^{[0]}=\la$ and $\la^{[l]}=\emptyset$.

\begin{thm}\label{t:iteGreen}
Let $\la$, $\mu$ be two partitions of $n$. Then we have
\begin{align}\label{e:iteGreen}
X^{\la}_{\mu}(q,t)=\sum_{\rho,\eta,\nu}\frac{c'_{\la}(q,t)(-1)^{|\nu/\eta|}|\la/\nu|}{c'_{\rho}(q,t)(1-t^{\mu_1})}t^{|\eta/\rho|}
\psi'_{\nu/\eta}(q,t)\varphi_{\la/\nu}(q,t)sk_{\eta/\rho}(q,t)X^{\rho}_{\mu^{[1]}}(q,t)
\end{align}
summed over all partitions $\rho,\eta,\nu$ such that $\la\supset\nu\supset\eta\supset\rho\vdash n-\mu_1$, $\la/\nu$ is a horizontal strip and $\nu/\eta$ is a vertical strip.
\end{thm}
\begin{proof}
By \eqref{e:pP}
\begin{align*}
p_{\mu^{[1]}}(X)=\sum_{\rho}\frac{\prod_{i=2}^{l(\mu)}(1-q^{\mu_i})}{c'_{\rho}(q,t)}X^{\rho}_{\mu^{[1]}}(q,t)P_{\rho}(X;q,t).
\end{align*}
Multiplying both sides by $p_{\mu_1}(X)$ implies that
\begin{align}\label{e:ppP}
p_{\mu}(X)=\sum_{\rho}\frac{\prod_{i=2}^{l(\mu)}(1-q^{\mu_i})}{c'_{\rho}(q,t)}X^{\rho}_{\mu^{[1]}}(q,t)p_{\mu_1}(X)P_{\rho}(X;q,t).
\end{align}
Combining \eqref{e:limit} and \eqref{e:M-M-N2}, we have
\begin{align*}
p_{\mu_1}(X)P_{\rho}(X;q,t)=\sum_{\la\vdash |\rho|+\mu_1}\frac{1-q^{\mu_1}}{1-t^{\mu_1}}\frac{\partial Q_{\la/\rho}(a-1;q,t)}{\partial a}\mid_{a=1}
P_{\la}(X;q,t),
\end{align*}
by which \eqref{e:ppP} can be rewritten as
\begin{align*}
p_{\mu}(X)=\sum_{\la,\rho}\frac{\prod_{i=1}^{l(\mu)}(1-q^{\mu_i})}{c'_{\rho}(q,t)(1-t^{\mu_1})}\frac{\partial Q_{\la/\rho}(a-1;q,t)}{\partial a}\mid_{a=1}X^{\rho}_{\mu^{[1]}}(q,t)P_{\la}(X;q,t).
\end{align*}
Note that we have
\begin{align*}
p_{\mu}(X)=\sum_{\la}\frac{\prod_{i=1}^{l(\mu)}(1-q^{\mu_i})}{c'_{\la}(q,t)}X^{\la}_{\mu}(q,t)P_{\la}(X;q,t).
\end{align*}
Since $\{P_{\la}(X;q,t)\}$ forms a basis, we derive that
\begin{align}\label{e:XX}
X^{\la}_{\mu}(q,t)=\sum_{\rho}\frac{c'_{\la}(q,t)}{c'_{\rho}(q,t)(1-t^{\mu_1})}\frac{\partial Q_{\la/\rho}(a-1;q,t)}{\partial a}\mid_{a=1}X^{\rho}_{\mu^{[1]}}(q,t).
\end{align}
Using \eqref{e:Q(a-1)}, we have
\begin{align}\label{e:partialQ}
\frac{\partial Q_{\la/\rho}(a-1;q,t)}{\partial a}\mid_{a=1}=\sum_{(\eta,\nu)}(-1)^{|\nu/\eta|}|\la/\nu|t^{|\eta/\rho|}\psi'_{\nu/\eta}(q,t)\varphi_{\la/\nu}(q,t)sk_{\eta/\rho}(q,t),
\end{align}
summed over all pairs of partitions $(\eta,\nu)$ such that $\rho\subset\eta\subset\nu\subset\la$, $\la/\nu$ is a horizontal strip and $\nu/\eta$ is a vertical strip.
Finally, substituting \eqref{e:partialQ} into \eqref{e:XX} gives \eqref{e:iteGreen}.
\end{proof}

We remark that there were
two iterative formulae for the Green polynomials $X^{\la}_{\mu}(t)$ with the help of the vertex operator realization of Hall-Littlewood functions. One is on $\la$ \cite[Theorem 2.6]{JL1}, the other is on $\mu$ \cite[Theorem 3.2]{JL1}. The one on $\mu$ can be regarded as the specialization of \eqref{e:iteGreen} ($q=0$).

It is known that the $(q, t)$-Kostka polynomials can be expressed in terms of the $(q, t)$-Green polynomials \cite[(8.20')]{Mac}:
\begin{equation}
K_{\mu\la}(q, t)=\sum_{\rho}z_{\rho}^{-1}\chi^{\mu}_{\rho}X^{\la}_{\rho}(q, t),
\end{equation}
where $\chi^{\mu}_{\rho}$ is the irreducible character $\chi^{\mu}$ of the symmetric group $\mathfrak S_n$ at the conjugacy of type $\rho$. Therefore
our formula \eqref{e:iteGreen} indirectly gives an iterative formula for the $(q, t)$-Kostka polynomials. We will return to derive two direct iterative formulae
for the $K_{\la\mu}(q, t)$ in Section \ref{s:Kostka}.

\section{Special cases of Corollary \ref{t:M-M-N2} }
In this section, we will consider some special cases of \eqref{e:M-M-N2}. We want to show how special cases of our rule recover various existing formulae. Particularly the special case for Schur functions (resp. Schur's $Q$-functions) is exactly the Murnaghan-Nakayama rule for the Hecke algebra (resp. Hecke-Clifford algebra), which explains why we call \eqref{e:M-M-N2} the Murnaghan-Nakayama rule for Macdonald functions.
\subsection{Hall-Littlewood functions}
First we introduce some needful notations. Recall the conjugate partition of $\la$ is denoted by $\la'$ and we write $m_{i}(\la)$ for the number of parts of $\la$ equal to $i$. The $q$-binomial coefficient is defined by
\begin{align*}
\left[\begin{matrix}a\\b\end{matrix}\right]_q=\frac{(q^{a-b+1}; q)_{b}}{(q; q)_b}=\frac{(1-q^a)(1-q^{a-1})\cdots(1-q^{a-b+1})}{(1-q^b)(1-q^{b-1})\cdots(1-q)}
\end{align*}
and is a polynomial in $q$ that reduces to the classical binomial coefficient $\left(\begin{matrix}a\\b\end{matrix}\right)$
when $q\to 1$.

It follows from the notations in the previous section that
\begin{align*}
\psi'_{\la/\mu}(t)&:=\psi'_{\la/\mu}(0,t)=\prod_{j\geq1}\left[\begin{matrix}\la'_j-\la_{j+1}'\\\la_{j}'-\mu_{j}'\end{matrix}\right]_t\\
sk_{\la/\mu}(t)&:=sk_{\la/\mu}(0,t)=t^{n(\la/\mu)}\prod_{j\geq1}
\left[\begin{matrix}\la'_j-\mu_{j+1}'\\m_{j}(\mu)\end{matrix}\right]_t
\end{align*}
Therefore,
\begin{align}
Q_{\la/\mu}(-1;t):=Q_{\la/\mu}(-1;0,t)=t^{|\la/\mu|}\sum_{\nu}(-t)^{-|\la/\nu|}\psi'_{\la/\nu}(t)sk_{\nu/\mu}(t),
\end{align}
summed over all partitions $\nu$ such that $\mu\subset\nu\subset\la$ and $\la/\nu$ is a vertical strip.

Now let us factorize $Q_{\la/\mu}(-1;t)$. Let $a_j=\la'_j-max(\mu'_j, \la'_{j+1})\geq0.$ Actually, $a_j$ is the number of those cells, which belong to $\la/\mu$ and are made of the rightmost boxes in the $j$th column in $\la$.  A partition $\nu$, $\mu\subset\nu\subset\la$, for which $\la/\nu$ is a vertical strip is obtained by choosing $k_j$ , $0 \leq k_j \leq a_j$ , and removing $k_j$ bottom cells of column $j$ in $\la$. See Fig. 1 for the example for $\la = (9,9,9,5,4,2,2)$ and $\mu=(5,5,4,3,1,1,1)$, where $a_2 = 2$, $a_4 = 1$, $a_5 = 1$, $a_9 = 3$, and $a_i = 0$ for all other $i$.
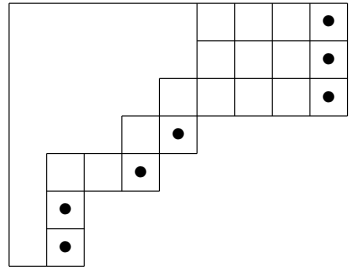
\begin{figure}
\begin{tikzpicture}[scale=0.5]
    \coordinate (Origin)   at (0,0);
    \coordinate (XAxisMin) at (0,0);
    \coordinate (XAxisMax) at (9,0);
    \coordinate (YAxisMin) at (0,0);
    \coordinate (YAxisMax) at (0,-9);
    \draw [thin, black] (4,0) -- (8,0);
    \draw [thin, black] (4,-1) --(8,-1);
    \draw [thin, black] (3,-2) -- (8,-2);
    \draw [thin, black] (2,-3) -- (8,-3);
    \draw [thin, black] (0,-4) -- (4,-4);
    \draw [thin, black] (0,-5) -- (3,-5);
    \draw [thin, black] (0,-6) -- (1,-6);
    \draw [thin, black] (0,-7) -- (1,-7);
     \draw [thin, black] (1,-4) -- (1,-7);
     \draw [thin, black] (2,-3) -- (2,-5);
     \draw [thin, black] (3,-2) -- (3,-5);
     \draw [thin, black] (4,0) -- (4,-4);
      \draw [thin, black] (5,0) -- (5,-3);
     \draw [thin, black] (6,0) -- (6,-3);
     \draw [thin, black] (7,0) -- (7,-3);
     \draw [thin, black] (8,0) -- (8,-3);
      \draw [thin, black] (0,-4) -- (0,-7);
      \node[inner sep=2pt] at (2.5,-4.5) {$\bullet$};
      \node[inner sep=2pt] at (0.5,-6.5) {$\bullet$};
      \node[inner sep=2pt] at (0.5,-5.5) {$\bullet$};
      \node[inner sep=2pt] at (3.5,-3.5) {$\bullet$};
      \node[inner sep=2pt] at (7.5,-2.5) {$\bullet$};
      \node[inner sep=2pt] at (7.5,-1.5) {$\bullet$};
      \node[inner sep=2pt] at (7.5,-0.5) {$\bullet$};
      \draw [thin, black] (-1,-7) -- (0,-7);
      \draw [thin, black] (-1,-7) -- (-1,0);
      \draw [thin, black] (-1,0) -- (4,0);
     \end{tikzpicture}
    \caption{A partition $\nu$ ($\mu\subset\nu\subset\la$) for which $\la/\nu$ is a
vertical strip within $\la/\mu$ is built
from $\la$ by removing some
number of the cells with bullets
of [$\la$].}
    \end{figure}

We have $|\la/\nu| = \sum_j k_j$ , $\nu'_{j}=\la'_j-k_j$. The choices of the $k_j$ are independent,
which means that
\begin{align}\label{e:Q(-1)}
\begin{split}
Q_{\la/\mu}(-1;t)&=t^{|\la/\mu|}\sum_{\nu}(-t)^{-|\la/\nu|}\psi'_{\la/\nu}(t)sk_{\nu/\mu}(t)\\
&=t^{|\la/\mu|}\sum_{k_1,k_2,\cdots}(-t)^{-\sum_j k_j}t^{n(\nu/\mu)}\prod_{j\geq1}
\left[\begin{matrix}\la'_j-\la_{j+1}'\\\la_{j}'-\nu_{j}'\end{matrix}\right]_t\prod_{j\geq1}
\left[\begin{matrix}\nu'_j-\mu_{j+1}'\\m_{j}(\mu)\end{matrix}\right]_t\\
&=t^{|\la/\mu|}\prod_j\left(\sum_{k_j=0}^{a_j}(-t)^{-k_j}t^{\left(\begin{matrix}\la_j'-\mu_{j}'-k_j\\2\end{matrix}\right)}
\left[\begin{matrix}m_j(\la)\\k_j\end{matrix}\right]_t\left[\begin{matrix}\la'_j-\mu_{j+1}'-k_j\\m_j(\mu)\end{matrix}\right]_t\right).
\end{split}
\end{align}

\subsection{Schur functions}
Recall a skew diagram $\theta$ is called a {\it border strip} if it is connected and contains no $2\times 2$ blocks (see \cite[p.5]{Mac}). A skew diagram is a {\it generalized border strip} if
it is a union of connected border strips.
Any generalized border strip is a union of its connected components, each of which is a border strip. A $k$-generalized border strip is a generalized border strip with $k$ boxes. In the example below, $\lambda=(4,3,3,1),$ $\mu=(3,2,1),$ the generalized border strip $\lambda/\mu$ has three border strips. It's a 5-generalized border strip.
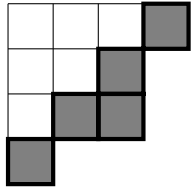
\begin{gather*}
  \centering
\begin{tikzpicture}[scale=0.6]
   \coordinate (Origin)   at (0,0);
    \coordinate (XAxisMin) at (0,0);
    \coordinate (XAxisMax) at (4,0);
    \coordinate (YAxisMin) at (0,-3);
    \coordinate (YAxisMax) at (0,0);
\draw [thin, black] (0,0) -- (4,0);
    \draw [thin, black] (0,-1) -- (4,-1);
    \draw [thin, black] (0,-2) -- (3,-2);
    \draw [thin, black] (0,-3) -- (3,-3);
    \draw [thin, black] (0,-4) -- (1,-4);
    \draw [thin, black] (0,0) -- (0,-4);
    \draw [thin, black] (1,0) -- (1,-4);
    \draw [thin, black] (2,0) -- (2,-3);
    \draw [thin, black] (3,0) -- (3,-3);
    \draw [thin, black] (4,0) -- (4,-1);
    \filldraw[fill = gray][ultra thick]
    (3,0) rectangle (4,-1) (2,-1)rectangle(3,-2) (2,-2)rectangle(3,-3) (1,-2)rectangle(2,-3) (0,-3)rectangle(1,-4);
    \end{tikzpicture}
\end{gather*}
If the generalized border strip $\theta=\la/\mu$ has $m$ connected components $(\xi_1,\xi_2,\ldots,\xi_m)$, we define the weight of $\theta$ by
$$wt(\theta;t)= (t-1)^{m-1}\prod\limits_{i=1}^{m}(-1)^{r(\xi_{i})-1}t^{c(\xi_i)-1}.$$
where the $r(\xi_i)$ (resp. $c(\xi_i)$) denotes the number of rows (resp. columns) in the border strip $\xi_i$. Usually we refer $\theta$ as a generalized border strip to emphasize the non-connectedness of $\theta$.

We have the following properties of Schur functions (in $\la$-ring notation) \cite[(6.2)]{HR}:
\begin{align}
\begin{split}
s_{\la/\mu}(X+Y)&=\sum_{\mu\subset\nu\subset\la}s_{\la/\nu}(X)s_{\nu/\mu}(Y) \quad \text{(sum rule)}\\
s_{\la/\mu}(-X)&=(-1)^{|\la/\mu|}s_{\la'/\mu'}(X) \quad \text{(duality rule)}\\
s_{\la/\mu}(zX)&=z^{|\la/\mu|}s_{\la/\mu}(X) \quad \text{(homogeneity)}.
\end{split}
\end{align}

The following result is similar to \cite[(6.7)]{HR}. We furnish a proof for completeness.
\begin{lem}\label{t:wt1}
In $\la$-ring natation, we have
\begin{align}
s_{\la/\mu}(a-1)=
\begin{cases}
(a-1)wt(\la/\mu;a) & \text{if $\la/\mu$ is a generalized border strip};\\
0 & \text{otherwise}.
\end{cases}
\end{align}
\end{lem}
\begin{proof}
By the sum rule and the duality rule,
\begin{align*}
s_{\la/\mu}(a-1)&=\sum_{\mu\subset\nu\subset\la}s_{\la/\nu}(a)s_{\nu/\mu}(-1)\\
&=\sum_{\mu\subset\nu\subset\la}s_{\la/\nu}(a)(-1)^{|\nu/\mu|}s_{\nu'/\mu'}(1).
\end{align*}
By \cite[p.71 (5.8)]{Mac}, for a singleton alphabet $x$, we have $s_{\rho/\tau}(x)=0$ unless $\rho/\tau$ is a horizontal strip, in which case $s_{\rho/\tau}(x)=x^{|\rho/\tau|}.$

Therefore,
\begin{align*}
s_{\la/\mu}(a-1)=\sum_{\mu\subset\nu\subset\la}a^{|\la/\nu|}(-1)^{|\nu/\mu|}.
\end{align*}
summed over all partitions $\nu$ such that $\la/\nu$ is a horizontal strip and $\nu/\mu$ is a vertical strip.

Hence $s_{\la/\mu}(a-1)=0$ unless $\la$ is obtained by the following method: first a vertical strip is added to $\mu$ to get $\nu$ and then a horizontal strip is added to $\nu$ to get $\la$ (see Fig. 2). That is, $s_{\la/\mu}(a-1)=0$ unless $\la/\mu$ is a generalized border strip.

Now we suppose $\la/\mu$ is a generalized border strip. Then each box in $\la/\mu$ satisfies one of the following:\\
(i) There is a box of $\la/\mu$ immediately below it (labeled by $v$);\\
(ii) There is a box of $\la/\mu$ immediately to its left (labeled by $h$);\\
(iii) Neither (i) nor (ii) holds (labeled by $*$).

As we observe, the following facts hold\\
(F1) Each box in case (i) (resp. (ii)) must come from the application of the vertical (resp. horizontal) strip to $\mu$ and thus this box has weight $-1$ (resp. $a$);\\
(F2) Each connected component (border strip) $\xi_i$ in $\la/\mu$ has $r(\xi_i)-1$ (resp. $c(\xi_i)-1$) boxes in case (i) (resp. (ii));\\
(F3) Each connected component (border strip) in $\la/\mu$ has exactly one box in case (iii), in which case this box has weight $a-1$.

Note that the product of these weights is exactly $(a-1)^{m}\prod\limits_{i=1}^{m}(-1)^{r(\xi_{i})-1}t^{c(\xi_i)-1}$, where $m$ is the number of connected components in $\la/\mu$ and $\xi_1,\cdots,\xi_m$ are connected components.
\begin{figure}
\begin{tikzpicture}[scale=0.5]
    \coordinate (Origin)   at (0,0);
    \coordinate (XAxisMin) at (0,0);
    \coordinate (XAxisMax) at (14,0);
    \coordinate (YAxisMin) at (0,0);
    \coordinate (YAxisMax) at (0,-14);
    \draw [thin, black] (0,0) -- (14,0);
    \draw [thin, black] (9,-1) --(14,-1);
    \draw [thin, black] (4,-2) -- (9,-2);
    \draw [thin, black] (4,-3) -- (8,-3);
    \draw [thin, black] (4,-4) -- (5,-4);
    \draw [thin, black] (4,-5) -- (5,-5);
    \draw [thin, black] (0,-6) -- (5,-6);
    \draw [thin, black] (0,-7) -- (1,-7);
     \draw [thin, black] (0,-8) -- (1,-8);
     \draw [thin, black] (0,-9) -- (1,-9);
     \draw [thin, black] (0,0) -- (0,-9);
     \draw [thin, black] (1,-6) -- (1,-9);
      \draw [thin, black] (4,-2) -- (4,-6);
     \draw [thin, black] (5,-2) -- (5,-6);
     \draw [thin, black] (6,-2) -- (6,-3);
     \draw [thin, black] (7,-2) -- (7,-3);
      \draw [thin, black] (8,-2) -- (8,-3);
      \draw [thin, black] (9,0) -- (9,-2);
      \draw [thin, black] (10,0) -- (10,-1);
      \draw [thin, black] (11,0) -- (11,-1);
      \draw [thin, black] (12,0) -- (12,-1);
      \draw [thin, black] (13,0) -- (13,-1);
      \draw [thin, black] (14,0) -- (14,-1);
      \node[inner sep=2pt] at (10.5,-0.5) {$h$};
      \node[inner sep=2pt] at (11.5,-0.5) {$h$};
      \node[inner sep=2pt] at (12.5,-0.5) {$h$};
      \node[inner sep=2pt] at (13.5,-0.5) {$h$};
      \node[inner sep=2pt] at (5.5,-2.5) {$h$};
      \node[inner sep=2pt] at (6.5,-2.5) {$h$};
      \node[inner sep=2pt] at (7.5,-2.5) {$h$};
      \node[inner sep=2pt] at (4.5,-3.5) {$v$};
      \node[inner sep=2pt] at (4.5,-4.5) {$v$};
      \node[inner sep=2pt] at (4.5,-5.5) {$v$};
      \node[inner sep=2pt] at (0.5,-6.5) {$v$};
      \node[inner sep=2pt] at (0.5,-7.5) {$v$};
      \node[inner sep=2pt] at (9.5,-0.5) {$*$};
      \node[inner sep=2pt] at (4.5,-2.5) {$*$};
      \node[inner sep=2pt] at (0.5,-8.5) {$*$};
     \end{tikzpicture}
    \caption{$\la$ is obtained from $\mu$ by adding a vertical strip (to get $\nu$) and then adding a horizontal strip (to $\nu$).}
    \end{figure}
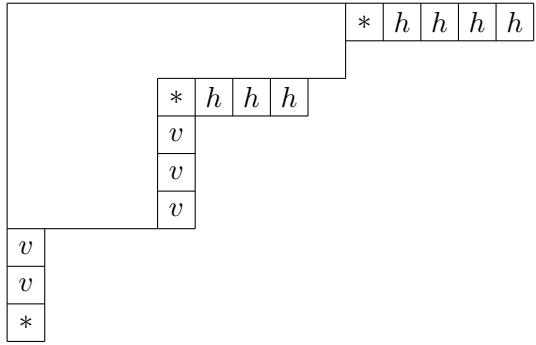
\end{proof}

Let $q=t=0$, we obtain symmetric functions $\tilde{g}_r(a-1;X;0,0)$ (denoted by $\tilde h_r(a)$) defined by its generating function:
\begin{align}
(a-1)\sum_{r\geq0}\tilde{h}_r(a)z^{r}=\prod_{i\geq1}\frac{1-x_iz}{1-ax_iz}.
\end{align}
Denote $\tilde{h}_{\la}(a)=\tilde{h}_{\la_1}(a)\tilde{h}_{\la_2}(a)\cdots$.

Ram \cite{Ram} used the quantum Schur-Weyl duality to prove the Frobenius type character formula
for the Hecke algebra $\mathcal H_n$ in terms of one-row Hall-Littlewood functions and Schur symmetric functions (see also \cite{KV, KW}), which says that
\begin{align}
\tilde{h}_{\mu}(q)=\sum\limits_{\lambda\vdash n}\chi^{\lambda}_{\mu}(q)s_{\lambda}.
\end{align}
where $\chi^{\la}_{\mu}(q)$ is the irreducible character of the Hecke algebra $H_n(q)$ indexed by $\la$ and evaluated at a certain
element $T_{\mu}$.

The following Murnaghan-Nakayama rule (see \cite{Ram, Hal, JL2}) for $H_n(q)$ is a special case of \eqref{e:M-M-N2}.
\begin{cor}
For each $r\geq0$ and each partition $\mu$,
\begin{align}
\tilde{h}_r(q)s_{\mu}=\sum_{\la}wt(\la/\mu;q)s_{\la}
\end{align}
summed over all partitions $\la$ such that $\la/\mu$ is an $r$-generalized border strip.
\end{cor}
\begin{proof} The result follows from Corollary \ref{t:M-M-N2} with $q\to 0, t\to 0$ and together with Lemma \ref{t:wt1}.
\end{proof}

\subsection{Schur's $Q$-functions}
A partition is called strict if its parts are distinct. For a strict partition $\lambda$, its {\it shifted diagram} $\lambda^*$ is obtained from the ordinary Young diagram by shifting the $k$th row to the right by $k-1$ squares, for each $k=1,2,\cdots,l(\la)$. A shifted skew diagram is a {\it shifted border strip} if its shifted diagram is connected and has no $2\times 2$ blocks. Given two strict partitions $\mu\subset\la$. We call $\la/\mu$ a {\it generalized strip} if the shifted skew diagram $\la^*/\mu^*$ has no $2\times 2$ block of squares. Particularly a connected generalized strip is exactly a shifted border strip. A skew diagram $\la/\mu$ is called a {\it double strip} if it is a union of two shifted border strips that both end on the main diagonal. A double strip can be cut into two non-empty connected pieces, one piece consisting of the diagonals of length $2$, and other piece consisting of the shifted border strip formed by the diagonals of length $1$. The following is an example for $\la=(6,5,4,3,2)$ and $\mu=(5,4,2)$. The boxes occupied by two red lines form respectively two shifted border strips. The piece of length $2$ (resp. $1$) is composed of all boxes marked by $\ast$' (resp. $\bullet$').
\begin{gather*}
  \centering
\begin{tikzpicture}[scale=0.6]
   \coordinate (Origin)   at (0,0);
    \coordinate (XAxisMin) at (0,0);
    \coordinate (XAxisMax) at (6,0);
    \coordinate (YAxisMin) at (0,5);
    \coordinate (YAxisMax) at (0,0);
    \draw [thin, black] (0,4) -- (0,5);
    \draw [thin, black] (1,3) -- (1,5);
    \draw [thin, black] (2,2) -- (2,5);
    \draw [thin, black] (3,1) -- (3,5);
    \draw [thin, black] (4,0) -- (4,5);
    \draw [thin, black] (5,0) -- (5,5);
    \draw [thin, black] (6,0) -- (6,5);
    \draw [thin, black] (4,0) -- (6,0);
    \draw [thin, black] (3,1) -- (6,1);
    \draw [thin, black] (2,2) -- (6,2);
    \draw [thin, black] (1,3) -- (6,3);
    \draw [thin, black] (0,4) -- (6,4);
    \draw [thin, black] (0,5) -- (6,5);
    \draw [thin, red] (3.5,1.5) -- (4.5,1.5);
    \draw [thin, red] (4.5,1.5) -- (4.5,2.5);
    \draw [thin, red] (4.5,2.5) -- (5.5,2.5);
    \draw [thin, red] (5.5,2.5) -- (5.5,4.5);
    \draw [thin, red] (4.5,0.5) -- (5.5,0.5);
    \draw [thin, red] (5.5,0.5) -- (5.5,1.5);
\end{tikzpicture}\qquad\qquad\qquad\qquad
\begin{tikzpicture}[scale=0.6]
   \coordinate (Origin)   at (0,0);
    \coordinate (XAxisMin) at (0,0);
    \coordinate (XAxisMax) at (6,0);
    \coordinate (YAxisMin) at (0,5);
    \coordinate (YAxisMax) at (0,0);
    \draw [thin, black] (0,4) -- (0,5);
    \draw [thin, black] (1,3) -- (1,5);
    \draw [thin, black] (2,2) -- (2,5);
    \draw [thin, black] (3,1) -- (3,5);
    \draw [thin, black] (4,0) -- (4,5);
    \draw [thin, black] (5,0) -- (5,5);
    \draw [thin, black] (6,0) -- (6,5);
    \draw [thin, black] (4,0) -- (6,0);
    \draw [thin, black] (3,1) -- (6,1);
    \draw [thin, black] (2,2) -- (6,2);
    \draw [thin, black] (1,3) -- (6,3);
    \draw [thin, black] (0,4) -- (6,4);
    \draw [thin, black] (0,5) -- (6,5);
    \node[inner sep=2pt] at (3.5,1.5) {$\ast$};
    \node[inner sep=2pt] at (4.5,1.5) {$\ast$};
    \node[inner sep=2pt] at (4.5,2.5) {$\ast$};
    \node[inner sep=2pt] at (5.5,1.5) {$\ast$};
    \node[inner sep=2pt] at (5.5,0.5) {$\ast$};
    \node[inner sep=2pt] at (4.5,0.5) {$\ast$};
    \node[inner sep=2pt] at (5.5,2.5) {$\bullet$};
    \node[inner sep=2pt] at (5.5,3.5) {$\bullet$};
    \node[inner sep=2pt] at (5.5,4.5) {$\bullet$};
\end{tikzpicture}
\end{gather*}
More generally, we call $\la/\mu$ a {\it generalized double strip} if there exists a strict partition $\mu\subset\nu\subset\la$ such that $\la^*/\nu^*$ and $\nu^*/\mu^*$ are both generalized strips. For a nonnegative integer $k$, a generalized double strip $\la/\mu$ is called a $k$-generalized double strip if $|\la/\mu|=k$.

\begin{exmp}\label{t:GDS}
$\la=(15,14,10,8,7,6,5,3,1)$, $\mu=(13,11,8,6,5,4,2,1)$ then $\la/\mu$ is a generalized double strip. Indeed we can choose $\nu=(15,13,10,8,6,5,4,2,1)$. The shifted diagram can be drawn as follows.
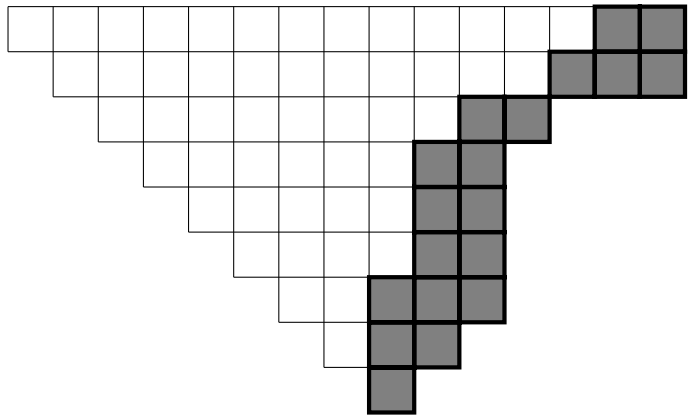
\begin{gather*}
  \centering
\begin{tikzpicture}[scale=0.6]
   \coordinate (Origin)   at (0,0);
    \coordinate (XAxisMin) at (0,0);
    \coordinate (XAxisMax) at (15,0);
    \coordinate (YAxisMin) at (0,0);
    \coordinate (YAxisMax) at (0,9);
    \draw [thin, black] (0,8) -- (0,9);
    \draw [thin, black] (1,7) -- (1,9);
    \draw [thin, black] (2,6) -- (2,9);
    \draw [thin, black] (3,5) -- (3,9);
    \draw [thin, black] (4,4) -- (4,9);
    \draw [thin, black] (5,3) -- (5,9);
    \draw [thin, black] (6,2) -- (6,9);
    \draw [thin, black] (7,1) -- (7,9);
    \draw [thin, black] (8,0) -- (8,9);
    \draw [thin, black] (9,0) -- (9,9);
    \draw [thin, black] (10,1) -- (10,9);
    \draw [thin, black] (11,2) -- (11,9);
    \draw [thin, black] (12,6) -- (12,9);
    \draw [thin, black] (13,7) -- (13,9);
    \draw [thin, black] (14,7) -- (14,9);
    \draw [thin, black] (15,7) -- (15,9);
    \draw [thin, black] (8,0) -- (9,0);
    \draw [thin, black] (7,1) -- (10,1);
    \draw [thin, black] (6,2) -- (11,2);
    \draw [thin, black] (5,3) -- (11,3);
    \draw [thin, black] (4,4) -- (11,4);
    \draw [thin, black] (3,5) -- (11,5);
    \draw [thin, black] (2,6) -- (12,6);
    \draw [thin, black] (1,7) -- (15,7);
    \draw [thin, black] (0,8) -- (15,8);
    \draw [thin, black] (0,9) -- (15,9);
    \filldraw[fill = gray][ultra thick]
    (8,1) rectangle (9,0) (8,2)rectangle(9,1) (8,2)rectangle(9,3) (9,1)rectangle(10,2) (9,2)rectangle(10,3) (9,3)rectangle(10,4) (9,4)rectangle(10,5) (9,5)rectangle(10,6) (10,2)rectangle(11,3) (10,3)rectangle(11,4) (10,4)rectangle(11,5) (10,5)rectangle(11,6) (10,6)rectangle(11,7) (11,6)rectangle(12,7) (12,7)rectangle(13,8) (13,7)rectangle(14,8) (13,8)rectangle(14,9) (14,7)rectangle(15,8) (14,8)rectangle(15,9);
\end{tikzpicture}
\end{gather*}
\end{exmp}
Similar to a double strip, a generalized double strip can be cut into two kinds of pieces (may be empty and not necessarily connected). One kind of pieces (denoted by $\alpha(\la/\mu)$) consist of the diagonals of length $2$, and the other kind (denoted by $\beta(\la/\mu)$) consisting of the generalized strips formed by the diagonals of length $1$. In the first kind of pieces, we call the box on the top left (resp. bottom right) $t$-box (resp. $-1$-box). We label the $t$-boxes (resp. $-1$-boxes) with $t$ (resp. $-1$) in $\la^*/\mu^*$. Clearly, the number of $t$-boxes and $-1$-boxes in $\alpha(\la/\mu)$ are same. Denote this number by $c(\la/\mu)$. $\beta(\la/\mu)$ is a union of its connected components, each of which is a shifted border strip. Denote the number of connected components in $\beta(\la/\mu)$ by $m(\la/\mu)$. The following graph is for Example \ref{t:GDS} and $c(\la/\mu)=5$, $m(\la/\mu)=5$.
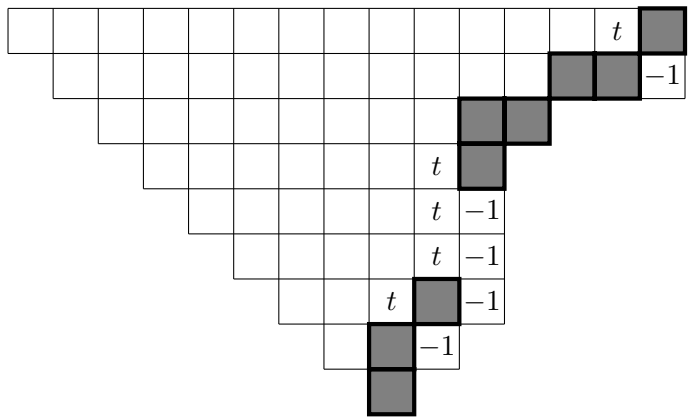
\begin{gather*}
  \centering
\begin{tikzpicture}[scale=0.6]
   \coordinate (Origin)   at (0,0);
    \coordinate (XAxisMin) at (0,0);
    \coordinate (XAxisMax) at (15,0);
    \coordinate (YAxisMin) at (0,0);
    \coordinate (YAxisMax) at (0,9);
    \draw [thin, black] (0,8) -- (0,9);
    \draw [thin, black] (1,7) -- (1,9);
    \draw [thin, black] (2,6) -- (2,9);
    \draw [thin, black] (3,5) -- (3,9);
    \draw [thin, black] (4,4) -- (4,9);
    \draw [thin, black] (5,3) -- (5,9);
    \draw [thin, black] (6,2) -- (6,9);
    \draw [thin, black] (7,1) -- (7,9);
    \draw [thin, black] (8,0) -- (8,9);
    \draw [thin, black] (9,0) -- (9,9);
    \draw [thin, black] (10,1) -- (10,9);
    \draw [thin, black] (11,2) -- (11,9);
    \draw [thin, black] (12,6) -- (12,9);
    \draw [thin, black] (13,7) -- (13,9);
    \draw [thin, black] (14,7) -- (14,9);
    \draw [thin, black] (15,7) -- (15,9);
    \draw [thin, black] (8,0) -- (9,0);
    \draw [thin, black] (7,1) -- (10,1);
    \draw [thin, black] (6,2) -- (11,2);
    \draw [thin, black] (5,3) -- (11,3);
    \draw [thin, black] (4,4) -- (11,4);
    \draw [thin, black] (3,5) -- (11,5);
    \draw [thin, black] (2,6) -- (12,6);
    \draw [thin, black] (1,7) -- (15,7);
    \draw [thin, black] (0,8) -- (15,8);
    \draw [thin, black] (0,9) -- (15,9);
    \node[inner sep=2pt] at (8.5,2.5) {$t$};
    \node[inner sep=2pt] at (9.5,3.5) {$t$};
    \node[inner sep=2pt] at (9.5,4.5) {$t$};
    \node[inner sep=2pt] at (9.5,5.5) {$t$};
    \node[inner sep=2pt] at (13.5,8.5) {$t$};
    \node[inner sep=2pt] at (9.5,1.5) {$-1$};
    \node[inner sep=2pt] at (10.5,2.5) {$-1$};
    \node[inner sep=2pt] at (10.5,3.5) {$-1$};
    \node[inner sep=2pt] at (10.5,4.5) {$-1$};
    \node[inner sep=2pt] at (14.5,7.5) {$-1$};
    \filldraw[fill = gray][ultra thick]
    (8,1) rectangle (9,0) (8,2)rectangle(9,1)   (9,2)rectangle(10,3)  (10,5)rectangle(11,6) (10,6)rectangle(11,7) (11,6)rectangle(12,7) (12,7)rectangle(13,8) (13,7)rectangle(14,8)  (14,8)rectangle(15,9);
\end{tikzpicture}
\end{gather*}

We remark that a generalized strip (or double strip) is naturally a generalized double strip. Let $\la/\mu$ be a generalized double strip, we have (1) if $c(\la/\mu)\geq1$, $m(\la/\mu)=1$, then $\la/\mu$ is a double strip; (2) if $c(\la/\mu)=0$, then $\la/\mu$ is a generalized strip; (3) if $c(\la/\mu)=0$, $m(\la/\mu)=1$, then $\la/\mu$ is a shifted border strip.

Let $\tau=(\tau_1,\tau_2,\cdots,\tau_a),\rho=(\rho_1,\rho_2,\cdots,\rho_b)$ be two compositions. We call $\rho$ is a {\it refinement} of $\tau$, denoted $\rho\prec\tau$, if there exists $i_0=0< i_1< i_2<\cdots< i_{a-1}< b=i_{a}$ such that $\tau_j=\rho_{i_{j-1}+1}+\rho_{i_{j-1}+2}+\cdots+\rho_{i_j}$, $j=1,2,\cdots,a$. For convenience, we also say that $\tau$ is a {\it coarsening} composition of $\rho$.
In particular, $\tau\prec\tau$. Note that a given composition has only finitely many coarsening ones.

Given a generalized double strip $\la/\mu$ with $\beta(\la/\mu)$ having $m$ connected components: $\xi^{(1)},\cdots,\xi^{(m)}$, we define its weight by
\begin{align}
\tilde{wt}(\la/\mu;t)=(-t)^{c(\la/\mu)}2^{\delta_{l(\la),l(\mu)+2}}\prod_{i=1}^{m}\left(\sum_{\xi^{(i)}\prec\tau}(-1)^{l(\xi^{(i)})-l(\tau)}d_{\tau}\right)
\end{align}
where $d_{\tau}=d_{\tau_1}d_{\tau_2}\cdots$ and $d_r=2(t-1)(-1)^{r-1}[r]_{-t}$.

\begin{lem}\label{t:duality}
We have the following properties of Schur's $Q$-functions (in $\la$-ring notation):
\begin{align}\label{e:duality}
\begin{split}
Q_{\la/\mu}(X+Y)&=\sum_{\mu\subset\nu\subset\la}Q_{\la/\nu}(X)Q_{\nu/\mu}(Y) \quad \text{\rm(sum rule)}\\
Q_{\la/\mu}(-X)&=(-1)^{|\la/\mu|}Q_{\la/\mu}(X) \quad \text{\rm(duality rule)}\\
Q_{\la/\mu}(zX)&=z^{|\la/\mu|}Q_{\la/\mu}(X) \quad \text{\rm(homogeneity)}.
\end{split}
\end{align}
\end{lem}
\begin{proof}
The first and the third identity can be proven similarly as those of Schur functions. The second identity is a result of \cite[p.259 Ex. 3]{Mac}.
\end{proof}

We remark that the duality rule in \eqref{e:duality} means that $Q_{\la/\mu}(-X)=Q_{\la/\mu}(-x_1,-x_2,\cdots)$, that is $Q_{\la/\mu}(-X)$ in $\la$-ring notation is exactly equal to $Q_{\la/\mu}(X)$ with $x_i$ replaced by $-x_i$. In particular, we have $Q_{\la/\mu}(t-1)=Q_{\la/\mu}(t,-1):=Q_{\la/\mu}(X)\mid_{x_1=t,x_2=-1,x_3=0,x_4=0,\cdots}$.

Let $q=0$, $t=-1$, we obtain symmetric function $\tilde{g}_r(a-1;X;0,-1)$ (denoted by $f_r(a)$) defined by its generating function:
\begin{align}
(a-1)\sum_{r\geq0}f_r(a)z^{r}=\prod_{i\geq1}\frac{1+ax_iz}{1-ax_iz}\frac{1-x_iz}{1+x_iz}.
\end{align}
Denote $f_{\la}(a)=f_{\la_1}(a)f_{\la_2}(a)\cdots$.

In \cite{WW}, Wan and Wang used the super-duality to derive a Frobenius character formula for the Hecke-Clifford algebra $\mathcal{H}^c_n$ in terms of spin Hall-Littlewood functions. Explicitly,
\begin{align}\label{e:characters}
f_{\mu}(q)=\sum\limits_{\lambda}2^{-\frac{l(\lambda)+\delta(\lambda)}{2}}\zeta^{\lambda}_{\mu}(q)Q_{\lambda}(X)
\end{align}
where $
\delta(\lambda)=
\begin{cases}
0 &\text {if $l(\lambda)$ is even,}\\
1 &\text {if $l(\lambda)$ is odd.}
\end{cases}$ And $\zeta^{\la}(q)$ is the irreducible character of the Hecke-Clifford algebra indexed by $\la$.

The following result is crucial to the Murnaghan-Nakayama rule for the Hecke-Clifford algebra. 
\begin{prop}\label{t:Qwt}\cite[Corollary 3.10]{JL3}
In the $\la$-ring language, we have
\begin{align}
Q_{\la/\mu}(t-1)=Q_{\la/\mu}(t,-1)=
\begin{cases}
\tilde{wt}(\la/\mu;t) &\text {if $\la/\mu$ is a generalized double strip,}\\
0 &\text {otherwise.}
\end{cases}
\end{align}
\end{prop}
The following Murnaghan-Nakayama rule for the Hecke-Clifford algebra is a special case of Corollary \ref{t:M-M-N2} ($q=0, t=-1$).
\begin{thm}
Let $r\geq0$, then
\begin{align}\label{e:Q-M-N}
f_r(q)Q_{\mu}(X)=\sum_{\la}(q-1)^{-1}\tilde{wt}(\la/\mu;q)Q_{\la}(X)
\end{align}
summed over all partitions $\la$ such that $\la/\mu$ is a generalized double strip.
\end{thm}
\begin{proof}
Taking $q=0$ and $t-1$, \eqref{e:M-M-N2} reduces to
\begin{align*}
f_{r}(q)Q_{\mu}(X)=\sum_{\la}(q-1)^{-1}Q_{\la/\mu}(q-1)Q_{\la}(X)
\end{align*}
Then \eqref{e:Q-M-N} follows from Proposition \ref{t:Qwt}.
\end{proof}

\begin{rem}
It would be interesting to find 
a representation theoretic explanation for our formula \eqref{e:M-M-N2} for arbitrary $q$ and $t$.
\end{rem}

\section{A combinatorial inversion of the Pieri rule for Hall-Littlewood functions}
In 2006, Lassalle and Schlosser \cite{LS} gave an analytic inversion of the Pieri formula for Macdonald polynomials, which asserts
\begin{align}\label{e:schlosser}
Q_{(\la_1,\cdots,\la_{n+1})}(X;q,t)=\sum_{\theta\in\mathbb{N}^n}C_{\theta_1,\cdots,\theta_n}^{(q,t)}(u_1,\cdots,u_n)
g_{\la_{n+1}-|\theta|}(X;q,t)Q_{(\la_1+\theta_1,\ldots,\la_n+\theta_n)}(X;q,t).
\end{align}
Here $u_i=q^{\la_i-\la_{n+1}}t^{n-i}$ and $C_{\theta_1,\theta_2,\cdots,\theta_n}^{(q,t)}(u_1,u_2,\cdots,u_n)$ is an explicit expression involving only $q$ and $t$ defined by
\begin{align}\label{e:C}
\begin{split}
C_{\theta_1,\theta_2,\cdots,\theta_n}^{(q,t)}(u_1,u_2,\cdots,u_n)=\prod_{k=1}^{n}t^{\theta_k}\frac{(q/t;q)_{\theta_k}}{(q;q)_{\theta_k}}
\frac{(qu_k;q)_{\theta_k}}{(qtu_k;q)_{\theta_k}}\prod_{1\leq i<j\leq n}\frac{(qu_i/tu_j;q)_{\theta_i}}{(qu_i/u_j;q)_{\theta_i}}
\frac{(tu_i/v_j;q)_{\theta_i}}{(u_i/v_j;q)_{\theta_i}}\\
\times \frac{1}{\bigtriangleup(v)}\det_{1\leq i,j\leq n}\left[v_i^{n-j}\left(1-t^{j-1}\frac{1-tv_i}{1-v_i}\prod_{k=1}^{n}\frac{u_k-v_i}{tu_k-v_i}\right)\right],
\end{split}
\end{align}
where $\bigtriangleup(v)$ is the Vandermonde determinant $\prod_{1\leq i,j\leq n}(v_i-v_j)$ and $v_i=q^{\theta_i}u_i$, $i=1,2,\cdots,n$.

It is worth pointing out that $(\la_1+\theta_1,\ldots,\la_n+\theta_n)$ is not necessarily a partition, in which case Macdonald polynomials are defined to be trivial. That is, $Q_{\la}(X;q,t)=0$ in \eqref{e:schlosser} when $\la$ is not a partition, which is different from the Hall-Littlewood case \cite[11.1]{LS}. For the reason that $C_{\theta}^{(q,t)}$ involves some determinant values, it is nontrivial to obtain explicitly $\lim_{q\rightarrow0}C_{\theta}^{(q,t)}$ (see \cite[Sec. 8]{LS} for details).

In this section, using Corollary \ref{t:M-M-N}, 
we will present a combinatorial inversion of the Pieri rule for Hall-Littlewood functions, which expands the Hall-Littlewood functions in terms of the product of a one-row Hall-Littlewood function and
 Hall-Littlewood functions indexed by partitions. Some special cases and applications of this formula are also considered.
\subsection{Vertex operator realization of Hall-Littlewood polynomials}
We restrict ourselves to the space $\Lambda_{t}=\Lambda\otimes\mathbb{Q}(t)$. Introduce the inner product on $\Lambda_t$ by
\begin{align}\label{e:Hinner}
\langle p_{\la}, p_{\mu} \rangle_{t}=\delta_{\la\mu}z_{\la}(0,t),
\end{align}
under which the Hall-Littlewood functions $Q_{\la}(X;t)$ and $P_{\la}(X;t)$ are dual to each other.

The multiplication operator $p_n:\Lambda_t\rightarrow\Lambda_t$ is of degree n. By \eqref{e:Hinner}, the dual operator is the differential operator
$p_n^*=\frac{n}{1-t^n}\frac{\partial}{\partial p_n}$ of degree $-n$. Note that $*$ is $\mathbb{Q}(t)$-linear and an anti-involution.

We now recall the vertex operator realization of the Hall-Littlewood symmetric functions from \cite{Jing1}.

The \textit{vertex operators} $H(z)$ are $t$-parameterized maps: $\Lambda_{t}\longrightarrow \Lambda_{t}[[z, z^{-1}]]$ defined by
\begin{align}
\label{e:hallop}
H(z)&=\mbox{exp} \left( \sum\limits_{n\geq 1} \dfrac{1-t^{n}}{n}p_nz^{n} \right) \mbox{exp} \left( -\sum \limits_{n\geq 1} \frac{\partial}{\partial p_n}z^{-n} \right)=\sum_{n\in\mathbb Z}H_nz^{n}.
\end{align}
The components $H_n$ are endomorphisms of $\Lambda_{t}$ of degree $n$, thus
$H_{-n}$ are annihilation operators for $n>0$. We collect their relations as follows.

\begin{prop}\cite{Jing1} The operators $H_n$ satisfy the following relations
\begin{align}\label{e:com1}
H_{m}H_n-tH_nH_m=tH_{m+1}H_{n-1}-H_{n-1}H_{m+1}, \quad H_{-n}. 1=\delta_{n, 0}.
\end{align}
where $\delta_{m, n}$ is the Kronecker delta function.
\end{prop}

We remark that the indexing of $H_m$ is different from that of \cite{Jing1}, where $H_n$ was denoted as $H_{-n}$ for instance.

\begin{prop} \cite{Jing1} \label{t:HL} Let $\lambda=(\lambda_{1},\ldots ,\lambda_{l})$ be a partition.
The vertex operator products  $H_{\lambda_{1}}\cdots H_{\lambda_{l}}. 1$ is the
Hall-Littlewood function $Q_{\la}(X;t)$:
\begin{equation}\label{e:HL}
H_{\lambda_{1}}\cdots H_{\lambda_{l}}. 1=Q_{\la}(X;t)=
\prod\limits_{i<j} \dfrac{1-R_{ij}}{1-tR_{ij}}q_{\lambda_{1}}\cdots q_{\lambda_{l}}
\end{equation}
where the raising operator $R_{ij}q_{\la}=q_{(\la_{1},\ldots ,\la_{i}+1,\ldots ,\la_{j}-1,\ldots , \la_{l})}$. Moreover,
$H_{\lambda}.1=H_{\lambda_1}\cdots H_{\lambda_l}.1$ are orthogonal in $\Lambda_t$:
\begin{align}\label{e:orth}
\langle H_{\lambda}.1, H_{\mu}.1\rangle_{t}=\delta_{\lambda\mu}b_{\lambda}(0,t),
\end{align}
where $b_{\lambda}(0,t)=(1-t)^{l(\lambda)}\prod_{i\geqslant 1}[m_i(\lambda)]_t!$, $[n]_t=\frac{1-t^n}{1-t}$ and $[n]_t!=[n]_t\cdots[1]_t$.
\end{prop}
\subsection{A combinatorial inversion of the Pieri rule}
We introduce the symmetric function $q_n = q_n(X;t)$ by the generating series
\begin{align}
\mbox{exp} \left( \sum\limits_{n\geq 1} \dfrac{1-t^{n}}{n}p_nz^{n} \right)=\sum_{n\geq0}q_n(X;t)z^n.
\end{align}
Clearly, the polynomial $q_n$ is the Hall-Littlewood symmetric function $Q_{(n)}$ associated with the one-row partition $(n)$. For any partition
$\la=(\la_1,\la_2,\cdots)$, we define $q_{\la}=q_{\la_1}q_{\la_2}\cdots$. Then the set $\{q_{\la}\}$ forms a basis of $\Lambda_t$, which are usually called the generalized homogenous polynomials.

Subsequently, using the $\la$-ring notation, we have
\begin{align}
\mbox{exp} \left( \sum\limits_{n\geq 1} \dfrac{t^{n}-1}{n}p_nz^{n} \right)=\sum_{n\geq0}q_n(-X;t)z^n.
\end{align}

Notice that $p_n^*=\frac{n}{1-t^n}\frac{\partial}{\partial p_n}$ with respect to $\langle \ , \ \rangle_t$, we have
\begin{align}
\mbox{exp} \left( -\sum \limits_{n\geq 1} \frac{\partial}{\partial p_n}z^{-n} \right)=\sum_{n\geq0}q^*_n(-X;t)z^{-n}.
\end{align}
Using \eqref{e:hallop}, we have
\begin{align}\label{e:decomposition}
H_n=\sum_{k\geq0}q_{n+k}(X;t)q^*_{k}(-X;t).
\end{align}
By \eqref{e:M-M-N}, we have
\begin{align}
q_{k}(-X;t)P_{\mu}(X;t)=\sum_{\la\vdash|\mu|+k}Q_{\la/\mu}(-1;t)P_{\la}(X;t).
\end{align}
By duality,
\begin{align}\label{e:q*Q}
q^*_{k}(-X;t)Q_{\la}(X;t)=\sum_{\mu\vdash|\la|-k}Q_{\la/\mu}(-1;t)Q_{\mu}(X;t).
\end{align}

Now we give another main result. Let us begin with some notations. Denote $\la^{[1]}=(\la_2,\la_3,\cdots,\la_l)$ and $C_{\la}^{i}=\{\mu\vdash i\mid \mu\subset\la\}$.

\begin{thm}\label{t:inverse}
For an arbitrary partition $\la=(\la_1,\la_2,\cdots,\la_l)\vdash n$, we have
\begin{align}\label{e:inverse}
Q_{\la}(X;t)=\sum_{k\geq0}^{n-\la_1}\sum_{\mu\in C_{\la^{[1]}}^{n-\la_1-k}}Q_{\la^{[1]}/\mu}(-1;t)q_{\la_1+k}(X;t)Q_{\mu}(X;t).
\end{align}
Here $Q_{\la^{[1]}/\mu}(-1;t)$ can be explicitly computed in a compact form by \eqref{e:Q(-1)}.
\end{thm}
\begin{proof} It follows from the vertex operator realization of Hall-Littlewood functions (Prop. \ref{t:HL}) that
\begin{align*}
Q_{\la}(X;t)&=H_{\la_1}H_{\la_2}\cdots H_{\la_l}.1  \\
&=\sum_{k\geq0}^{n-\la_1}q_{\la_1+k}(X;t)q^*_{k}(-X;t)H_{\la^{[1]}}.1 \quad \text{(by \eqref{e:decomposition})}\\
&=\sum_{k\geq0}^{n-\la_1}q_{\la_1+k}(X;t)\sum_{\mu\in C_{\la^{[1]}}^{n-\la_1-k}}Q_{\la^{[1]}/\mu}(-1;t)Q_{\mu}(X;t) \quad \text{(by \eqref{e:q*Q})}\\
&=\sum_{k\geq0}^{n-\la_1}\sum_{\mu\in C_{\la^{[1]}}^{n-\la_1-k}}Q_{\la^{[1]}/\mu}(-1;t)q_{\la_1+k}(X;t)Q_{\mu}(X;t).
\end{align*}
\end{proof}

\subsection{Special cases of Theorem \ref{t:inverse}}
Taking $t=0$ and $t=-1$ respectively, \eqref{e:inverse} reduces to
\begin{align}\label{e:sinverse}
s_{\la}(X)=\sum_{k=0}^{n-\la_1}\sum_{\mu\in C^{n-\la_1-k}_{\la^{[1]}}}s_{\la^{[1]}/\mu}(-1)s_{(\la_1+k)}(X)s_{\mu}(X)
\end{align}
and
\begin{align}\label{e:Qinverse}
Q_{\la}(X)=\sum_{k=0}^{n-\la_1}\sum_{\mu\in C^{n-\la_1-k}_{\la^{[1]}}}Q_{\la^{[1]}/\mu}(-1)Q_{(\la_1+k)}(X)Q_{\mu}(X).
\end{align}

\begin{lem}\label{t:-1}
Let $\mu\subset\la$ be two partitions, then
\begin{align}
s_{\la/\mu}(-1)&=
\begin{cases}
(-1)^{|\la/\mu|}& \text{if $\la/\mu$ is a vertical strip,}\\
0& \text{otherwise.}
\end{cases}\\
Q_{\la/\mu}(-1)&=
\begin{cases}
(-1)^{|\la/\mu|}& \text{if $\la/\mu$ is a horizontal strip,}\\
0& \text{otherwise.}
\end{cases}
\end{align}
\end{lem}
\begin{proof}
By duality, we have
\begin{align*}
s_{\la/\mu}(-1)=(-1)^{|\la/\mu|}s_{\la'/\mu'}(1).
\end{align*}
Note that $s_{\rho/\tau}(1)=0$ unless $\rho/\tau$ is a horizontal strip, in which case $s_{\rho/\tau}(-1)=(-1)^{|\rho/\tau|}.$ This yields the first identity. The case of Schur's $Q$-functions is similar. Just note $Q_{\la/\mu}(-1)=(-1)^{|\la/\mu|}Q_{\la/\mu}(1)$.
\end{proof}

Using Lemma \ref{t:-1}, \eqref{e:sinverse} and \eqref{e:Qinverse} can be rewritten respectively as:
\begin{align}\label{e:sdecom}
s_{\la}(X)=\sum_{\mu}(-1)^{|\la^{[1]}/\mu|}s_{(|\la|-|\mu|)}(X)s_{\mu}(X)
\end{align}
summed over all partitions $\mu$ such that $0\leq|\mu|\leq n-\la_1$ and $\la^{[1]}/\mu$ is a vertical strip and

\begin{align}
Q_{\la}(X)=\sum_{\mu}(-1)^{|\la^{[1]}/\mu|}Q_{(|\la|-|\mu|)}(X)Q_{\mu}(X)
\end{align}
summed over all strict partitions $\mu$ such that $0\leq|\mu|\leq n-\la_1$ and $\la^{[1]}/\mu$ is a horizontal strip.

\section{Two iterative formulae for the $(q,t)$-Kostka polynomial}\label{s:Kostka}
In this section, using the dual version of our multiparametric Murnaghan-Nakayama rule for Macdonald polynomials \eqref{e:dual}, we will give two iterative formulae for the $(q,t)$-Kostka polynomial. The first one expresses the $(q,t)$-Kostka polynomial in terms of those with smaller lengths of the upper partition, while the second one reduces to those with shorter lengths of the lower partition. Moreover, using the first iterative formula, we will derive a general formula for the $(q,t)$-Kostka polynomial via the Lassalle-Okounkov generalized $(q, t)$-binomial coefficients.

On the one hand, choose $\mathbf{a}=(1,t,t^2,t^3,\cdots)$, $\mathbf{b}=\emptyset$ in \eqref{e:dual}. Then we have
\begin{align}\label{e:g*Q}
g^{\bot}_k(\frac{X}{1-t};q,t)Q_{\la}(X;q,t)=\sum_{\mu}Q_{\la/\mu}(\frac{1}{1-t};q,t)Q_{\mu}(X;q,t)
\end{align}
where
\begin{align}
\sum_{k=0}^{\infty}g^{\bot}_{k}(\frac{X}{1-t};q,t)z^k=\exp\left(\sum_{n=1}^{\infty}\frac{1}{1-t^n}\frac{\partial}{\partial p_n(X)}z^n\right).
\end{align}

For $q,t\rightarrow0$ followed by $\mathbf{a}=(1,q,q^2,q^3,\cdots)$ and $\mathbf{b}=\emptyset$, in this order, \eqref{e:dual} reduces to
\begin{align}\label{e:g*Q2}
g^{\bot}_k(\frac{X}{1-q};0,0)s_{\la}(X)=\sum_{\mu}s_{\la/\mu}(\frac{1}{1-q})s_{\mu}(X)
\end{align}
where
\begin{align}
\sum_{k=0}^{\infty}g^{\bot}_{k}(\frac{X}{1-q};0,0)z^k=\exp\left(\sum_{n=1}^{\infty}\frac{1}{1-q^n}\frac{\partial}{\partial p_n(X)}z^n\right).
\end{align}

On the other hand, recall that the {\it complete symmetric function} $s_{(n)}$ (one-row Schur function) is defined by
\begin{align}
\sum_{n\geq0}s_{(n)}(X)z^n=\exp\left(\sum_{n\geq0}\frac{1}{n}p_n(X)z^n\right).
\end{align}
Note that $p^*_n=\frac{n}{1-t^n}\frac{\partial}{\partial p_n}$ with respect to $\langle \cdot , \cdot \rangle_t$, therefore the adjoint operator
$s^*_{(n)}$ is given by
\begin{align}
\sum_{n\geq0}s^*_{(n)}z^n=\exp\left(\sum_{n\geq0}\frac{1}{1-t^n}\frac{\partial}{\partial p_n}z^n\right).
\end{align}
Similarly, \eqref{e:defg} implies the following adjoint operator:
\begin{align}
\sum_{k=0}^{\infty}g^{*}_{k}(X;q,t)z^k=\exp\left(\sum_{n=1}^{\infty}\frac{1}{1-q^n}\frac{\partial}{\partial p_n(X)}z^n\right).
\end{align}
Thus $s^*_{(k)}(X)=g^{\bot}_k(\frac{X}{1-t};q,t)$ and $g^{*}_{k}(X;q,t)=g^{\bot}_{k}(\frac{X}{1-q};0,0)$. Here we emphasize that the former pair are
the adjoint operators (denoted by $*$) with respect to the inner product $\langle \ , \ \rangle_t$, while the latter pair are the adjoint operators
(denoted by $\perp$) are relative to the inner product $\langle \ , \ \rangle_{q,t}$.

If we introduce a set of (fictitious) variables $\xi_i$ by means of
\begin{align}
\prod_i\frac{1-tx_iy}{1-x_iy}=\prod_i\frac{1}{1-\xi_iy},
\end{align}
then the {\it big Schur function} $S_{\la}(X;t)$ is defined by $S_{\la}(X;t):=s_{\la}(\xi)$, which is also
$s_{\la}(\frac{X}{1-t})$ in $\lambda$-ring notation. It is known that the bases $\{S_{\la}(X;t)\}$ and $\{s_{\la}(X)\}$ are dual to each other under the inner product $\langle \cdot ,\cdot \rangle_t$ defined in \eqref{e:Hinner}.

Now we are ready to formulate our final results. We start by recalling the $(q,t)$-Kostka polynomials $K_{\la \mu}(q,t)$ defined by \cite[(8.11)]{Mac}
\begin{align}
J_{\mu}(X;q,t)=\sum_{\la}K_{\la \mu}(q,t)S_{\la}(X;t),
\end{align}
where $J_{\mu}(X;q,t)$ is the integral Macdonald function associated to partition $\mu$, $|\mu|=|\la|$.

The duality of $\{s_{\la}(X)\}$ and $\{S_{\la}(X; t)\}$ gives that
\begin{align}\label{e:Kinner}
K_{\la \mu}(q,t)=\langle J_{\mu}(X;q,t), s_{\la}(X) \rangle_t=c_{\mu}'(q,t)\langle Q_{\mu}(X;q,t), s_{\la}(X) \rangle_t.
\end{align}
\subsection{The first iteration}
\begin{thm}\label{t:q,t-Kite}
Let $\la,\mu$ be two arbitrary partitions of the same weight, we have
\begin{align}\label{e:q,t-Kite}
K_{\la \mu}(q,t)=\sum_{\rho,\tau}\frac{c_{\mu}'(q,t)}{c_{\tau}'(q,t)}(-1)^{|\la^{[1]}/\rho|}Q_{\mu/\tau}(\frac{1}{1-t};q,t)K_{\rho \tau}(q,t)
\end{align}
summed over all partitions $\rho$ and $\tau$ such that $\rho\subset\la^{[1]}$, $\tau\subset\mu$ and $\la^{[1]}/\rho$ is a vertical strip.
\end{thm}
\begin{proof} It follows from direct calculation that
\begin{align*}
K_{\la \mu}(q,t)=&c'_{\mu}(q,t)\left\langle Q_{\mu}(X;q,t), s_{\la}(X) \right\rangle_t \quad \text{(by \eqref{e:Kinner})}\\
=&c'_{\mu}(q,t)\left\langle Q_{\mu}(X;q,t), \sum_{\rho}(-1)^{|\la^{[1]}/\rho|}s_{(|\la|-|\rho|)}s_{\rho}(X) \right\rangle_t \quad \text{(by \eqref{e:sdecom})}\\
=&\sum_{\rho}(-1)^{|\la^{[1]}/\rho|}c'_{\mu}(q,t)\left\langle s^*_{(|\la|-|\rho|)}Q_{\mu}(X;q,t), s_{\rho}(X) \right\rangle_t,
\end{align*}
summed over $\rho$ such that $\la^{[1]}/\rho$ is a vertical strip.
Note that $s^*_{(k)}(X)=g^{\bot}_{k}(\frac{X}{1-t};q,t)$. It follows from \eqref{e:g*Q} that
\begin{align*}
K_{\la \mu}(q,t)=&\sum_{\rho}(-1)^{|\la^{[1]}/\rho|}c'_{\mu}(q,t)\left\langle \sum_{\tau}Q_{\mu/\tau}(\frac{1}{1-t};q,t)Q_{\tau}(X;q,t), s_{\rho}(X) \right\rangle_t\\
=&\sum_{\rho,\tau}\frac{c_{\mu}'(q,t)}{c_{\tau}'(q,t)}(-1)^{|\la^{[1]}/\rho|}Q_{\mu/\tau}(\frac{1}{1-t};q,t)K_{\rho \tau}(q,t),
\end{align*}
summed over $\rho$ and $\tau$ such that $\tau\supset\mu$ and $\la^{[1]}/\rho$ is a vertical strip.
\end{proof}


Now we will show $\frac{c_{\mu}'(q,t)}{c_{\tau}'(q,t)}Q_{\mu/\tau}(\frac{1}{1-t};q,t)$ is exactly the generalized $(q, t)$-binomial coefficients $\left(\begin{matrix}\mu\\ \tau\end{matrix}\right)_{q,t}$ (up to some power of $t$) introduced independently by Lassalle and Okounkov \cite{L2,Ok}.

To see this we introduce some notations. Let $u$ be a composition with length $n$, i.e., $u=(u_1,\cdots,u_n)\in\mathbb{N}^n$. The symmetric group $\mathfrak{S}_n$ acts on compositions by permuting the parts. The unique partition in the $\mathfrak{S}_n$-orbit of $u$ is denoted by $u^+$, which
 is usually called the sorted partition of $u$. Write $u^+=\sigma_u(u)$ for the unique permutation of minimal length. 
 The staircase partition $\delta$ is defined as $\delta=(n-1,n-2,\cdots,1,0)$. We also denote $a^{u}b^v=(a^{u_1}b^{v_1},a^{u_2}b^{v_2},\cdots,a^{u_n}b^{v_n})$ for scalars $a,b$ and compositions $u,v$ with length $n$.

The {\it spectral vector} $\langle u\rangle$ of $u$ is defined by $\langle u\rangle:=q^ut^{\sigma_u(\delta)}$. For example, let $u=(3,4,2,1,0,2,5)$, then $u^+=(5,4,3,2,2,1,0)$, $\sigma_u=(1753)(46)$, $\langle u\rangle=(q^3t^4,q^4t^5,q^2t^2,qt,1,q^2t^3,q^5t^6)$.

Let $x=(x_1,x_2,\cdots,x_n)$. The {\it interpolation Macdonald polynomial}  $M_{u}(x):=M_{u}(x;q,t)$ is defined \cite{Kno,Lasc2,Sa1,Sa2} as the unique polynomial of degree $|u|$ such that \\
(i) $M_u(\langle v\rangle)=0$ for $|v|\leq|u|$, $u\neq v$; \\
(ii) The coefficient of $x^u$ is $q^{-n'(u)}$, where $n'(u)=n((u^+)')$.

Introduce the {\it nonsymmetric Macdonald polynomial} $E_{u}(x)$ by $$E_{u}(x)=E_{u}(x;q,t):=q^{n'(u)}\lim_{a\rightarrow0}a^{|u|}M_{u}(x/a).$$

For convenience, define the normalized Macdonald polynomials:
\begin{align*}
\mathbf{M}_{u}(x):=q^{n'(u)}t^{n(u)}\frac{M_{u}(x)}{c'_{u}(q,t)},\quad \mathbf{E}_{u}(x):=t^{n(u)}\frac{E_{u}(x)}{c'_{u}(q,t)}.
\end{align*}

The symmetric analogue of $\mathbf{M}_{u}(x)$, denoted by $\mathbf{MS}_{u}(x)$ (called {\it symmetric interpolation polynomials} or {\it shifted Macdonald polynomials}), is given by
\begin{align*}
\mathbf{MS}_{\la}(x):=\sum_{u^+=\la}\mathbf{M}_{u}(x), \quad \text{$\la$ is a partition}.
\end{align*}

Note that some definitions above are not standard but provide more succinct description for our purposes.

We introduce two generalized $(q, t)$-binomial coefficients:
\begin{align*}
\left[\begin{matrix}u\\ v\end{matrix}\right]_{q,t}:=\frac{\mathbf{M}_{v}(\langle u\rangle)}{\mathbf{M}_{v}(\langle v\rangle)},\quad\quad\quad
\left(\begin{matrix}\la\\ \mu\end{matrix}\right)_{q,t}:=\frac{\mathbf{MS}_{\mu}(\langle \la\rangle)}{\mathbf{MS}_{\mu}(\langle \mu\rangle)}.
\end{align*}
The first one was introduced by Sahi \cite{Sa2}. As the symmetric analogue of the first one, the second one was introduced independently by Lassalle and Okounkov \cite{L2,Ok}. These two generalized $(q, t)$-binomial coefficients both have nice orthogonality relations:
\begin{align}
\sum_{v}\frac{\tau_{v}}{\tau_{u}}\left[\begin{matrix}u\\ v\end{matrix}\right]_{q,t}\left[\begin{matrix}v\\ w\end{matrix}\right]_{q^{-1},t^{-1}}&=\delta_{uw}\\
\sum_{\mu}\frac{\tau_{\mu}}{\tau_{\la}}\left(\begin{matrix}\la\\ \mu\end{matrix}\right)_{q,t}\left(\begin{matrix}\mu\\ \nu\end{matrix}\right)_{q^{-1},t^{-1}}&=\delta_{\la\nu},
\end{align}
where $\tau_{u}=(-1)^{|u|}q^{n'(u)}t^{-n(u^+)}$.

The relation between these two generalized $(q, t)$-binomial coefficients can be stated as follows \cite[(3.46)]{BF}
\begin{align}\label{e:rel}
\left(\begin{matrix}\la\\ \mu\end{matrix}\right)_{q,t}=\sum_{v^+=\mu}\left[\begin{matrix}u\\ v\end{matrix}\right]_{q,t}
\end{align}
where $u$ is any composition such that $u^+=\la$.

It follows from \cite[pp. 633-634]{LRW} that
\begin{align}\label{e:Qlamu}
\frac{c_{\la}'(q,t)}{c_{\mu}'(q,t)}Q_{\la/\mu}(\frac{1}{1-t};q,t)=t^{n(\la)-n(\mu)}\sum_{v^+=\mu}\mathbf{E}_{u/v}(1,0).
\end{align}
where $u$ is any composition such that $u^+=\la$.

Combining \eqref{e:rel}, \eqref{e:Qlamu} and \cite[6.7a]{LRW} we have
\begin{align}
\mathbf{E}_{u/v}(1,0)=\left[\begin{matrix}u\\ v\end{matrix}\right]_{q,t},
\end{align}
subsequently
\begin{align}
\frac{c_{\la}'(q,t)}{c_{\mu}'(q,t)}Q_{\la/\mu}(\frac{1}{1-t};q,t)=t^{n(\la)-n(\mu)}\left(\begin{matrix}\la\\ \mu\end{matrix}\right)_{q,t}.
\end{align}

Therefore we can rewrite \eqref{e:q,t-Kite} as follows:
\begin{align}\label{e:q,t-KBite}
K_{\la \mu}(q,t)=\sum_{\rho,\tau}(-1)^{|\la^{[1]}/\rho|}t^{n(\mu)-n(\tau)}\left(\begin{matrix}\mu\\ \tau\end{matrix}\right)_{q,t}K_{\rho \tau}(q,t)
\end{align}
summed over all partitions $\rho$ and $\tau$ such that $\tau\subset\mu$ and $\la^{[1]}/\rho$ is a vertical strip. Now let's say that
$\la\sqsupset \rho$ if $\la^{[1]}/\rho$ is a vertical strip.

By a straightforward iteration of \eqref{e:q,t-KBite} we obtain a general formula for the $(q,t)$-Kostka polynomial via the Lassalle-Okounkov generalized $(q, t)$-binomial coefficient.
\begin{cor}\label{t:Kgeneral}
Let $\la,\mu$ be two arbitrary partitions with the same weight, then we have
\begin{align}\label{e:Kgeneral}
K_{\la\mu}(q,t)=(-1)^{|\la|}t^{n(\mu)}\sum_{\underline{\la}, \underline{\mu}}(-1)^{wid(\underline{\la})}
\prod_{i=0}^{r-1}\left(\begin{matrix}\mu^{(i)}\\ \mu^{(i+1)}\end{matrix}\right)_{q,t}
\end{align}
summed over sequences of partitions
$\underline{\la}: \la=\la^{(0)}\sqsupset \la^{(1)} \sqsupset \cdots\sqsupset\la^{(r)}=\emptyset$ and $\underline{\mu}: \mu=\mu^{(0)}\supset \mu^{(1)} \supset \cdots\mu^{(r)}=\emptyset$ ($1\leq r\leq l(\la)$) such that $|\la^{(i)}|=|\mu^{(i)}|$, $i=1,2,\cdots r-1$ ( i.e. $(\la^{(i)})^{[1]}/\la^{(i+1)}$ is a vertical strip). Here
$wid(\underline{\la})=\sum_{i\geq 0} wid(\la^{(i)})$ and $wid(\la)=\la_1$ is the width of partition $\lambda$.
\end{cor}

We remark that for the case of $q=0$, we have a closed form for $Q_{\mu/\tau}(\frac{1}{1-t};q,t)$ (see \cite[Theor. 3.1]{Lasc} or \cite[(4.3)]{WZ})
\begin{align}
Q_{\mu/\tau}(\frac{1}{1-t};t)=sk_{\mu/\tau}(t):=t^{n(\mu/\tau)}\prod_{j\geq1}
\left[\begin{matrix}\mu'_j-\tau_{j+1}'\\m_{j}(\tau)\end{matrix}\right]_t,
\end{align}
by which an explicit iterative formula for the Kostka polynomial $K_{\la \mu}(t)$ is derived:
\begin{cor}
(i) As a special case of Theorem \ref{t:q,t-Kite}, we have
\begin{align}\label{e:Kite}
K_{\la \mu}(t)=\sum_{\rho,\tau}(-1)^{|\la^{[1]}/\rho|}t^{n(\mu/\tau)}\prod_{j\geq1}
\left[\begin{matrix}\mu'_j-\tau_{j+1}'\\\tau'_j-\tau'_{j+1}\end{matrix}\right]_tK_{\rho\tau}(t)
\end{align}
summed over all partitions $\rho$ and $\tau$ such that $\rho\subset\la^{[1]}$, $\tau\subset\mu$ and $\la^{[1]}/\rho$ is a vertical strip.

(ii) Explicitly, for two partitions $\la, \mu$ of equal weight
\begin{equation}\label{e:iterK}
K_{\la \mu}(t)=(-1)^{|\la|}\sum_{\underline{\la}, \underline{\mu}}(-1)^{wid(\underline{\la})}t^{\sum_{i\geq 0}n(\mu^{(i)}/\mu^{(i+1)})}\prod_{i\geq 0}\prod_{j\geq1}
\left[\begin{matrix}\mu^{(i)'}_j-\mu^{(i+1)'}_{j+1}\\\mu^{(i+1)'}_j-\mu^{(i+1)'}_{j+1}\end{matrix}\right]_t
\end{equation}
summed over sequences of partitions
$\underline{\la}: \la=\la^{(0)}\sqsupset \la^{(1)} \sqsupset \cdots\la^{(r)}=\emptyset$ and $\underline{\mu}: \mu=\mu^{(0)}\supset \mu^{(1)} \supset \cdots\mu^{(r)}=\emptyset$ ($1\leq r\leq l(\la)$) such that $(\la^{(i)})^{[1]}/\la^{(i+1)}$ is a vertical strip and $|\la^{(i)}|=|\mu^{(i)}|$, $i=1,2,\cdots r-1$.
\end{cor}

In \cite[Corollary 2.8]{BJ}, Bryan and the first author expressed $K_{\la \mu}(t)$ in terms of those with lower partitions of smaller lengths. 
Whereas our \eqref{e:Kite} involves those indexed by upper partitions with smaller lengths.  Kirillov and Reshetikhin \cite{KR} also gave a
formula for $K_{\la\mu}(t)$ in terms of $t$-Gaussian numbers and Bethe Ansatz, which looks similar but different from our formula \eqref{e:iterK}.

\subsection{The second iteration}

\begin{thm}\label{t:q,t-Kite2}
Let $\la$ and $\mu$ be two partitions with the same weight, $l(\mu)=l+1$, then
\begin{align}
K_{\la\mu}(q,t)=\sum_{\theta\in\mathbb{N}^l}\sum_{\rho\vdash|\la|-\mu_{l+1}+|\theta|}\frac{c_{\mu}'(q,t)}{c_{\tau}'(q,t)}
C_{\theta_1,\theta_2,\cdots,\theta_l}^{(q,t)}(u_1,u_2,\cdots,u_l)s_{\la/\rho}(\frac{1}{1-q})K_{\rho,\mu^{[1]}+\theta}(q,t),
\end{align}
where $u_i=q^{\mu_i-\mu_{l+1}}t^{l-i}$ and $C_{\theta_1,\theta_2,\cdots,\theta_l}^{(q,t)}(u_1,u_2,\cdots,u_l)$ is defined in \eqref{e:C}. $s_{\la/\rho}(\frac{1}{1-q})$ is computed explicitly in \eqref{e:first} and \eqref{e:second}.
\end{thm}
\begin{proof}
It follows from direct calculation that
\begin{align*}
&K_{\la \mu}(q,t)\\
=&c'_{\mu}(q,t)\left\langle Q_{\mu}(X;q,t), s_{\la}(X) \right\rangle_t \quad \text{(by \eqref{e:Kinner})}\\
=&c'_{\mu}(q,t)\left\langle \sum_{\theta\in\mathbb{N}^l}C_{\theta_1,\theta_2,\cdots,\theta_l}^{(q,t)}(u_1,u_2,\cdots,u_l)
g_{\mu_{l+1}-|\theta|}(X;q,t)Q_{(\mu_1+\theta_1,\ldots,\mu_l+\theta_l)}(X;q,t), s_{\la}(X) \right\rangle_t \quad \text{(by \eqref{e:schlosser})}\\
=&\sum_{\theta\in\mathbb{N}^l}c'_{\mu}(q,t)C_{\theta_1,\theta_2,\cdots,\theta_l}^{(q,t)}(u_1,u_2,\cdots,u_l)\left\langle
Q_{(\mu_1+\theta_1,\ldots,\mu_l+\theta_l)}(X;q,t), g^*_{\mu_{l+1}-|\theta|}(X;q,t)s_{\la}(X) \right\rangle_t.
\end{align*}
Note that $g^{*}_{k}(X;q,t)=g^{\bot}_{k}(\frac{X}{1-q};0,0)$. Combining this with \eqref{e:g*Q2} gives
\begin{align*}
&K_{\la \mu}(q,t)\\
=&\sum_{\theta\in\mathbb{N}^l}c'_{\mu}(q,t)C_{\theta_1,\theta_2,\cdots,\theta_l}^{(q,t)}(u_1,u_2,\cdots,u_l)\left\langle
Q_{(\mu_1+\theta_1,\ldots,\mu_l+\theta_l)}(X;q,t), \sum_{\rho\vdash |\la|-\mu_{l+1}+|\theta|}s_{\la/\rho}(\frac{1}{1-q})s_{\rho}(X) \right\rangle_t\\
=&\sum_{\theta\in\mathbb{N}^l}\sum_{\rho\vdash|\la|-\mu_{l+1}+|\theta|}\frac{c_{\mu}'(q,t)}{c_{\tau}'(q,t)}
C_{\theta_1,\theta_2,\cdots,\theta_l}^{(q,t)}(u_1,u_2,\cdots,u_l)s_{\la/\rho}(\frac{1}{1-q})K_{\rho,\mu^{[1]}+\theta}(q,t).
\end{align*}
\end{proof}
Now we present some formulas for $s_{\la/\rho}(\frac{1}{1-q})$ for completeness of Theorem \ref{t:q,t-Kite2}. 
Let $T$ be a {\it standard Young tableaux} (written as SYT in short) with $n$ boxes (see \cite[p. 5]{Mac} for explicit definition). We say that
$1\leq i\leq n$ is a descent of $T$ if $i+1$ appears in a 
lower row than
that of $i$. The major index $maj(T)$ of $T$ is the sum of all descents
of $T$. The {\it fake-degree polynomial} $f^{\la/\rho}(q)$ associated with a skew Young diagram $\la/\rho$ is defined as the generating function of $maj(T)$:
\begin{align*}
f^{\la/\rho}(q)=\sum_{T} q^{maj(T)}
\end{align*}
summed over all SYTs of skew shape $\la/\rho$.

Let $\la/\rho$ be a skew shape with $n$ boxes. Then the first formula of $s_{\la/\mu}(\frac1{1-q})$ says that \cite[Prop. 7.19.11]{St}
\begin{align}\label{e:first}
s_{\la/\rho}(\frac{1}{1-q})=\frac{f^{\la/\rho}(q)}{(1-q)(1-q^2)\cdots(1-q^n)}.
\end{align}

In particular, for $\la=(n)$, $\rho=\emptyset$, we have $maj(T)=0$ for any SYT of shape $(n)$. Thus, $s_{(n)}(\frac{1}{1-q})=\frac{1}{(1-q)(1-q^2)\cdots(1-q^n)}=\frac{1}{(q;q)_n}$.

Using the Jacobi-Trudi formula for skew Schur functions \cite[p.70, (5.4)]{Mac}:
\begin{align*}
s_{\la/\rho}=\det_{1\leq i,j\leq l(\la)}\left(s_{(\la_i-\rho_j-i+j)}\right),
\end{align*}
we have the second formula
\begin{align}\label{e:second}
s_{\la/\rho}(\frac{1}{1-q})=\det_{1\leq i,j\leq l(\la)}\left(\frac{1}{(q;q)_{\la_i-\rho_j-i+j}}\right).
\end{align}

\vskip30pt \centerline{\bf Acknowledgments}
We thank Michael J. Schlosser for helpful discussions on the inversion of the Pieri formula for Macdonald polynomials. The paper is partially supported by Simons Foundation grant
MP-TSM-00002518
and NSFC grant No. 12171303. The second author is supported by China Scholarship Council.
\bigskip

\bibliographystyle{plain}

\end{document}